\documentclass[11pt]{amsart}
\usepackage{hyperref}
\usepackage{amsfonts}
\usepackage{amsmath}
\usepackage{amssymb}
\usepackage{amscd}
\usepackage{graphicx}
\usepackage{enumitem}
\usepackage{latexsym}
\usepackage{color}
\usepackage{comment}
\usepackage{epsfig}
\usepackage{amsopn}
\usepackage{mathrsfs}
\usepackage{array}
\usepackage[usenames,dvipsnames]{pstricks}
\usepackage{epsfig}
\usepackage{pst-grad} 
\usepackage{pst-plot} 
\usepackage[space]{grffile} 
\usepackage{etoolbox} 
\usepackage{algorithm}
\usepackage[noend]{algorithmic}

\usepackage{enumitem}
\newlist{steps}{enumerate}{1}
\setlist[steps, 1]{wide=0pt, leftmargin=\parindent, label=Step \arabic*:, font=\bfseries}

\makeatletter 
\patchcmd\Gread@eps{\@inputcheck#1 }{\@inputcheck"#1"\relax}{}{}
\makeatother

\hypersetup{pdfpagemode=UseNone}
\setenumerate[0]{label=(\alph*)}

\usepackage{booktabs}
\usepackage{longtable}
\theoremstyle{plain}
\newtheorem{lemma}{Lemma}[section]
\newtheorem*{theorem*}{Theorem}
\newtheorem*{lemma*}{Lemma}
\newtheorem*{proposition*}{Proposition}
\newtheorem*{conjecture*}{Conjecture}
\newtheorem*{corollary*}{Corollary}
\newtheorem*{problem*}{Problem}
\newtheorem{theorem}[lemma]{Theorem}

\newtheorem{corollary}[lemma]{Corollary}
\newtheorem{proposition}[lemma]{Proposition}

\newtheorem{problem}[lemma]{Problem}

\theoremstyle{definition}
\newtheorem{definition}[lemma]{Definition}
\newtheorem{example}[lemma]{Example}
\newtheorem{remark}[lemma]{Remark}

\DeclareMathOperator{\sHom}{\mathcal{H}\kern -.5pt\mathit{om}}
\DeclareMathOperator{\sTor}{\mathcal{T}\kern -1.5pt\mathit{or}}

\begin{document}

\title{The Rigidity Problem in Orthogonal Grassmannians}

\author[Y. Liu]{Yuxiang Liu}
\address{Department of Mathematics, Statistics and CS \\University of Illinois at Chicago, Chicago, IL 60607}
\email{yliu354@uic.edu}


\begin{abstract}
We classify rigid Schubert classes in orthogonal Grassmannians. More generally, given a representative $X$ of a Schubert class in an orthogonal Grassmannian, we give combinatorial conditions which guarantee that every linear space parametrized by $X$ meets a fixed linear space in the required dimension.
\end{abstract}

\maketitle

\setcounter{tocdepth}{1}
\tableofcontents

\section{Introduction}
In this paper, we study the rigidity problem for orthogonal Grassmannians. In particular, given a representative $X$ of a Schubert class in an orthogonal Grassmannian, we give combinatorial conditions which guarantee that every linear space parametrized by $X$ meets a fixed linear space in the required dimension. We first introduce the necessary notation and state our results. 

Let $V$ be a $n$-dimensional complex vector space and let $q$ be a nonsingular symmetric bilinear form on $V$. If $n\neq 2k$, let $OG(k,n)$ denote the orthogonal Grassmannian that parametrizes $k$-dimensional isotropic subspaces of $V$. If $n=2k$, then the space of $k$-dimensional isotropic subspaces has two irreducible components, and we let $OG(k,2k)$ denote one of the components. 

\begin{definition}Given two increasing sequences of integers $a=(a_1,...,a_s)$ and $b=(b_1,...,b_{k-s})$ such that
\begin{eqnarray}
1&\leq& a_1<...< a_s\leq \frac{n}{2},\nonumber\\
0&\leq& b_1<...< b_{k-s}\leq \frac{n}{2}-1,\nonumber
\end{eqnarray}
 where $1\leq s\leq k$ and such that $a_i\neq b_j+1$ for all $1\leq i\leq s,1\leq j\leq k-s$, and a flag of isotropic subspaces $F_\bullet=F_1\subset...\subset F_{\left[n/2\right]}$, where the lower indices indicate the vector space dimension, the corresponding Schubert variety $\Sigma_{a;b}(F_\bullet)$ is defined to be the Zariski closure of the following locus in $OG(k,n)$:
$$\{\Lambda\in OG(k,n)|\dim(\Lambda\cap F_{a_i})= i, \dim(\Lambda\cap F_{b_j}^\perp)= k-j+1, 1\leq i\leq s, 1\leq j\leq k-s\}.$$
\end{definition}
The Schubert class $\sigma_{a;b}$ is the cohomology class of $\Sigma_{a;b}(F_\bullet)$, which is independent of the choice of $F_\bullet$. It is called rigid if the only representatives are Schubert varieties. Given a Schubert class, we ask the following question: 
\begin{problem}
Let $X$ be a subvariety of $OG(k,n)$ representing the Schubert class $\sigma_{a;b}$. When does there exist an isotropic subspace $F_{a_i}$ (or $F_{b_j}$) such that for every $ \Lambda\in X$, $\dim(\Lambda\cap F_{a_i})\geq i$ (or $\dim(\Lambda\cap F_{b_j}^\perp)\geq k-j+1)$?
\end{problem}

To the best of the author's knowledge, this paper is the first investigation of this generalized problem in orthogonal Grassmannians. Our results in particular fully characterize the rigid Schubert classes in orthogonal Grassmannians. 

First, we study this problem in Grassmannians. Let $G(k,n)$ denote the Grassmannian variety that parametrizes $k$-dimensional subspaces of a $n$-dimensional complex vector space $V$. Given a partial flag of subspaces $$F_\bullet=F_{a_1}\subsetneq...\subsetneq F_{a_k}\subset V,\ \  \dim(F_{a_i})=a_i,$$ the Schubert variety $\Sigma_{a_1,...,a_k}(F_\bullet)$ is defined to be the following locus in $G(k,n)$:
$$\Sigma_{a_1,...,a_k}(F_\bullet)=\{\Lambda\in G(k,n)|\dim(\Lambda\cap F_{a_i})\geq i,1\leq i\leq k\}.$$
Let $\sigma_{a_1,...,a_k}$ denote its cohomology class.

In \S \ref{sec1}, we will prove the following result:
\begin{theorem}\label{index}
Let $\sigma_{a_1,...,a_k}$ be a Schubert class in $G(k,n)$. Assume for some $1\leq i\leq k$, $a_{i+1}\neq a_i+1$ and one of the following holds:
\begin{enumerate}
\item $i=k$; or

\item $a_i=i$; or

\item $a_i\leq a_{i+1}-3$; or

\item $a_i=a_{i-1}+1$.
\end{enumerate}

Then for every subvariety $X$ representing the Schubert class $\sigma_{a_1,...,a_k}$ in $G(k,n)$, there exists a unique $a_i$-dimensional linear subspace $F_{a_i}$ of $V$ such that $$\dim(\Lambda\cap F_{a_i})\geq i,\ \ \ \forall \Lambda \in X.$$
\end{theorem}
This theorem vastly generalizes Coskun's rigidity results in \cite{Coskun2011RigidAN}. Even for a non-rigid Schubert class, this theorem describes the minimal Schubert variety (up to a general translate) that contains every representative of a given Schubert class.

\begin{example}
Let $X$ be a general subvariety representing the Schubert class $\sigma_{1,3,5}$ in $G(3,6)$. Then by Theorem \ref{index}, there exist fixed linear spaces $F_1$ and $F_5$ of dimension $1$ and $5$ respectively, such that every $3$-plane parametrized by $X$ contains $F_1$ and is contained in $F_5$. Therefore $X$ is contained in the Schubert variety $$\Sigma_{1,4,5}=\{\Lambda\in G(3,6)|\dim(\Lambda\cap F_1)\geq 1,\dim(\Lambda\cap F_5)\geq 3\}.$$
\end{example}
We then extend the ideas to orthogonal Grassmannians. We define the {\em rigid} sub-index as follows:

\begin{definition}
Let $\sigma_{a;b}$ be a Schubert class in $OG(k,n)$. A sub-index $a_i$, $i<s$ is called {\em essential} if $a_{i+1}\neq a_i+1$. The sub-index $a_s$ is called {\em essential} unless $n=2k$ and $a_s=b_{k-s}+2=k$. An essential sub-index $a_i$ is called {\em rigid} if for every subvariety $X$ of $OG(k,n)$ representing $\sigma_{a;b}$, there exists an isotropic subspace $F_{a_i}$ of dimension $a_i$ such that $$\dim(\Lambda\cap F_{a_i})\geq i,\ \ \ \forall\Lambda\in X.$$
Similarly, a sub-index $b_j$ is called {\em essential} if $b_{j-1}\neq b_j-1$. An essential sub-index $b_j$ is called {\em rigid} if for every subvariety $X$ representing $\sigma_{a;b}$, there exists an isotropic subspace $F_{b_j}$ of dimension $b_j$ such that $$\dim(\Lambda\cap F_{b_j}^\perp)\geq j\ \ \ \forall\Lambda\in X.$$
\end{definition}

In \S \ref{sec3}, we characterize the rigid sub-indices. We summarize our results as follows:
\begin{theorem}\label{main theorem}
Let $\sigma_{a;b}$ be a Schubert class in $OG(k,n)$. An essential sub-index $a_i$ is not rigid if and only if one of the following holds:
\begin{enumerate}
\item $a_i\neq b_j$ for all $1\leq j\leq k-s$, $a_i-a_{i-1}\geq 2$ and $$a_{i+1}-a_i=2+\#\{j|a_i< b_j<a_{i+1}\};$$
\item $a_i=b_j$ for some $j$ and $$\#\{\mu|a_\mu\leq b_j\}=k-j+b_{j}-\frac{n-3}{2}.$$
\end{enumerate}

An essential sub-index $b_j$ is rigid if and only if either $b_j=0$ or there exist $1\leq i\leq s$ and $j\leq j'\leq k-s$ such that $a_i=b_{j'}$ and $$\#\{\mu|a_{\mu}\leq b_{j'}\}>k-j'+b_{j'}-\frac{n-3}{2}.$$

\end{theorem}
As an application, we fully characterize the rigid Schubert classes in orthogonal Grassmannians.

\begin{theorem}\label{rigidclass}
Let $\sigma_a^b$ be a Schubert class in $OG(k,n)$. Let $b_\gamma$ be the largest essential sub-index in $b=(b_1,...,b_{k-s})$. Then $\sigma_{a;b}$ is rigid if and only if all of the following conditions hold:
\begin{enumerate}
\item $b_\gamma=a_i$ for some $1\leq i\leq s$ and $$\#\{\mu|a_{\mu}\leq b_{\gamma}\}>k-\gamma+b_{\gamma}-\frac{n-3}{2};$$
\item there is no $1\leq i\leq s$ such that $a_i\neq b_j$ for all $1\leq j\leq s$, $a_i-a_{i-1}\geq 2$ and $$a_{i+1}-a_i=2+\#\{j|a_i< b_j<a_{i+1}\}.$$
\end{enumerate}
\end{theorem}

There is a stronger type of rigidity. A Schubert class is called {\em Schur rigid} (or {\em multi-rigid}) if every multiple of the Schubert class can only be represented by a union of Schubert varieties. One way to approach Schur rigidity problems is to use differential systems. Walters and Byrant studied this problem by transforming it to the problem on the integral varieties of differential systems and proved the rigidity of certain homology classes \cite{Walter}, \cite{RB2000}. Hong characterized the Schur rigidity of smooth Schubert classes in Hermitian symmetric spaces and some singular Schubert varieties in Grassmannians\cite{Ho1}, \cite{Ho2}. Robles and The extended this method and characterized the Schur rigid classes in irreducible compact Hermitian symmetric space \cite{RT}. Using the geometric theory of uniruled projective manifolds, Hong and Mok proved the rigidity \cite{HM1} and Schur rigidity \cite{Hong2020SchurRO} of smooth Schubert varieties in rational homogeneous spaces of Picard number one. Using algebro-geometric methods, Coskun characterized Schur rigid classes in Grassmannians and proved a large set of classes is not Schur rigid in orthogonal Grassmannians \cite{Coskun2014RigidityOS}. 

The rigidity problem arise from the smoothability problem, which asks whether a cohomology class has a smooth representative. The smoothability problems are well investigated \cite{Coskun2011RigidAN},\cite{Coskun2014RigidityOS}, \cite{KL1},\cite{RB2000}. If a Schubert class is rigid and singular, then it is non-smoothable. Our results prove certain classes in orthogonal Grassmannian are not smoothable.

\subsection*{Organization of the paper} In \S \ref{Preliminaries}, we review some basic facts about the Schubert classes and introduce the rigid sub-indices. In \S \ref{sec1}, we characterize rigid sub-indices in the case of Grassmannians. In \S \ref{sec2}, we review the basic facts of restriction varieties and recall an algorithm to compute the cohomology class of a restriction variety. In \S \ref{sec3}, we characterize rigid sub-indices in orthogonal Grassmannians. As a corollary, we characterize rigid Schubert classes in $OG(k,n)$.

\subsection*{Acknowledgments} The author would like to thank Izzet Coskun for invaluable discussions and support.

\section{Preliminaries}\label{Preliminaries}
In this section, we recall the necessary definitions regarding the Schubert varieties and the rigidity problem in Grassmannians and orthogonal Grassmannians.

\subsection{Grassmannians}
Let $V$ be an $n$-dimensional complex vector space and let $G(k,n)$ be the Grassmannian variety parameterizing all $k$-dimensional linear subspaces of $V$. A {\em Schubert index} $a=(a_1,...,a_k)$ is an increasing sequence of $k$ positive integers $$1\leq a_1<...< a_k\leq n.$$ 

\begin{definition}
Given a Schubert index $a=(a_1,...,a_k)$ and a partial flag $F_\bullet=F_{a_1}\subset...\subset F_{a_k}$ of subspaces of $V$, $\dim(F_{a_i})=a_i$,
the Schubert variety $\Sigma_a(F_\bullet)$ is defined to be the following closed subset in $G(k,n)$:
$$\Sigma_a(F_\bullet):=\{\Lambda\in G(k,n)|\dim(\Lambda\cap F_{a_i})\geq i, \ 1\leq i\leq k\}.$$
\end{definition}
The Schubert variety $\Sigma_a(F_\bullet)$ is irreducible of dimension $\sum a_i-i$ (see \cite[Theorem 4.1]{3264}). In the Chow ring of $G(k,n)$, we define the Schubert classes to be the rational equivalence class of the Schubert varieties:
$$\sigma_a(F_\bullet):=[\Sigma_a(F_\bullet)]\in A(G(k,n)).$$
The Chow ring of $G(k,n)$ is generated by the Schubert classes (see \cite[Corollary 4.7]{3264}). Since any two flags differ by an element of $GL_n(V)$, the class $\sigma_a(F_\bullet)$ is independent of the choice of partial flags $F_\bullet$. Hence we will omit $F_\bullet$ and write $\sigma_a=\sigma_a(F_\bullet)$. 

\begin{remark}
It is customary to index a Schubert class by a sequence of non-increasing integers $n-k\geq \lambda_1\geq...\geq\lambda_k\geq 0$. The index $a$ can be translated into $\lambda$ by $\lambda_i=n-k+i-a_i$. The advantage of our notation is that $a$ is invariant under the natural inclusion $G(k,n)\hookrightarrow G(k,n+1)$. This notation can also be easily adapted to orthogonal Grassmannians.
\end{remark}

\begin{definition}
Given a Schubert index $a=(a_1,...,a_k)$, an {\em essential subindex} is a subindex $a_i$ such that $a_{i}\neq a_{i+1}-1$. A {\em rigid subindex} is an essential subindex such that for every subvariety $X$ representing the Schubert class $\sigma_a$, there exists a unique $a_i$-dimensional linear space $F_{a_i}$ such that $$\dim(\Lambda\cap F_{a_i})\geq i,\ \ \ \forall \Lambda \in X.$$
\end{definition}

\begin{definition}
A Schubert class is called {\em rigid} if the only representatives are Schubert varieties. 
\end{definition}
\begin{remark}
It is clear that for a rigid Schubert class, all the essential sub-indices are rigid. We will later show that the converse is also true in \S\ref{sec1}.
\end{remark}

\subsection{Orthogonal Grassmannians}Let $V$ be an $n$-dimensional complex vector space and let $q$ be a non-degenerate symmetric bilinear form on $V$. A linear space $W$ is called isotropic with respect to $q$ if $q(W,W)=0$.  For $k<\frac{n}{2}$, the orthogonal Grassmannian $OG(k,V)=OG(k,n)$ is the subvariety of $G(k,n)$ that parametrizes all isotropic $k$-subspaces with respect to $q$. If $n$ is even and $k=\frac{n}{2}$, then the space of $k$-dimensional isotropic subspaces has two irreducible components, and we let $OG(k,2k)$ denote one of the components. The variety $OG(k,n)$, considered as a subvariety of $G(k,n)$, has dimension $k(n-k)-\binom{k+1}{2}$ (see \cite[Proposition 4.15]{3264}).

Given an isotropic subspace $W$, we denote $W^\perp$ its orthogonal complement with respect to $q$. Now fix a complete flag of isotropic subspaces $F_\bullet=F_1\subset...\subset F_{\left[n/2\right]}$. When $n$ is even, the maximal dimensional isotropic subspaces of $V$ have dimension $\frac{n}{2}$ and form two connected components. The orthogonal complement of $F_{n/2-1}$ is a union of two maximal isotropic subspaces which belong to the different components. We use the convention that the maximal isotropic subspace in the different component than $F_{n/2}$ is denoted by $F_{n/2-1}^\perp$. 
\begin{definition}\label{Schubert variety}
A Schubert index in the orthogonal Grassmannian consists of two increasing sequences of integers $1\leq a_1<...< a_s\leq \frac{n}{2}$ and $0\leq b_1<...< b_{k-s}\leq \frac{n}{2}-1$, where $1\leq s\leq k$ and such that $a_i\neq b_j+1$ for all $1\leq i\leq s,1\leq j\leq k-s$. Given a Schubert index $(a;b)$ and an isotropic flag $F_\bullet$, we define the Schubert variety $\Sigma_{a;b}(F_\bullet)$ to be the Zariski closure of the following locus in $OG(k,n)$:
$$\Sigma_{a;b}(F_\bullet):=\{\Lambda\in OG(k,n)|\dim(\Lambda\cap F_{a_i})= i, \dim(\Lambda\cap F_{b_j}^\perp)= k-j+1, 1\leq i\leq s, 1\leq j\leq k-s\}.$$
\end{definition}
The Schubert classes $[\Sigma_{a;b}(F_\bullet)]\in A(OG(k,n))$ generate the Chow ring $A(OG(k,n)$ and are independent of the choice of flags $F_\bullet$ \cite{Bra08}. We will omit $F_\bullet$ and denote them by $\sigma_{a;b}=[\Sigma_{a;b}(F_\bullet)]$.

\begin{remark}
Here is an explanation of why we require $a_i\neq b_j+1$ for all $i,j$ in the definition of Schubert indices. If $a_i=b_j+1$ for some $i,j$, then for every $k$-plane $\Lambda$ parametrized by $\Sigma_{a;b}(F_\bullet)$, there are two possibilities:
\begin{enumerate}
\item $\dim(\Lambda\cap F_{a_{i-1}})=\dim(\Lambda\cap F_{a_i})$
\item $\dim(\Lambda\cap F_{a_{i-1}})=\dim(\Lambda\cap F_{a_i})-1$
\end{enumerate}
All $k$-planes satisfying condition $(a)$ form the Schubert variety $\Sigma_{a'}^{b'}$, where $a'$ is obtained from $a$ by replacing $i$ with $i-1$ and $b'=b$. Now consider a $k$-plane $\Lambda$ satisfying condition $(b)$. Choose $v\in \Lambda\cap F_{a_i}$ such that $v\notin F_{a_{i-1}}$. Then $F_{a_i}$ is the span of $F_{a_{i-1}}$ and $v$. Since $\Lambda\subset v^\perp$ and $F_{a_i}^\perp=F_{a_{i-1}}^\perp\cap v^\perp$, $\dim(\Lambda\cap F^\perp_{a_{i-1}})=\dim(\Lambda\cap F^\perp_{i})$. We get all $k$-planes satisfying $(b)$ form the Schubert variety $\Sigma_{a'';b''}$ where $a''=a$ and $b''$ is obtained from $b$ by replacing $i-1$ with $i$. Therefore if $a_i=b_j+1$, then $\Sigma_{a;b}$ is a union of two Schubert varieties.
\end{remark}

\begin{definition}\label{def1}
Let $\sigma_{a;b}$ be a Schubert class in $OG(k,n)$. A sub-index $a_i$, $i<s$ is called {\em essential} if $a_{i+1}\neq a_i+1$. The sub-index $a_s$ is called {\em essential} unless $n=2k$ and $a_s=b_{k-s}+2=k$. An essential sub-index $a_i$ is called {\em rigid} if for every subvariety $X$ of $OG(k,n)$ representing $\sigma_{a;b}$, there exists a unique isotropic subspace $F_{a_i}$ of dimension $a_i$ such that $$\dim(\Lambda\cap F_{a_i})\geq i,\ \ \ \forall\Lambda\in X.$$
Similarly, a sub-index $b_j$ is called {\em essential} if $b_{j-1}\neq b_j-1$. An essential sub-index $b_j$ is called {\em rigid} if for every subvariety $X$ representing $\sigma_{a;b}$, there exists a unique isotropic subspace $F_{b_j}$ of dimension $b_j$ such that $$\dim(\Lambda\cap F_{b_j}^\perp)\geq j\ \ \ \forall\Lambda\in X.$$
\end{definition}
In \S\ref{sec3}, we will show that a Schubert class in $OG(k,n)$ is rigid if and only if all essential indices are rigid.

\section{Rigidity problems in Grassmannians}\label{sec1}
In this section, we prove Theorem \ref{index}. As a corollary, we characterize all the rigid classes in the Grassmannian $G(k,n)$.

We recall the following two propositions which are first proved by Coskun \cite{Coskun2011RigidAN}.
\begin{proposition}
\label{contained in plane}\cite[Proposition 3.1]{Coskun2011RigidAN}
Let $X\subset G(k,n)$ be a subvariety with class $[X]=\sigma_a$. Then there exists a fixed $a_k$-dimensional space $F_{a_k}$ such that $\Lambda\subset F_{a_k}$ for all $\Lambda\in X$.
\end{proposition}

\begin{proof}
Let $Z$ be the projective variety swept out by projective $(k-1)$-planes parametrized by $X$. Suppose $\dim(Z)=s$ and $\deg(Z)=d$. It suffices to show that $s=a_k-1$ and $d=1$. 

Let $\mathbb{P}(G_{n-s})=\mathbb{P}^{n-s-1}$ be a general projective linear space of dimension $n-s-1$ in $\mathbb{P}^{n-1}$. Then $\mathbb{P}(G_{n-s})$ will meet $Z$ in $d$ points $p_1,...,p_d$. Let $\Lambda\in X$ be a linear space such that $p_1\in\mathbb{P}(\Lambda)$. Then $\Lambda$ is contained in the intersection 
$$X\cap\Sigma_{n-s,n-k+2,...,n-1,n}$$
where $\Sigma_{n-s,n-k+2,...,n-1,n}$ is the Schubert variety which parametrizes all $k$ planes that meet $G_{n-s}$ in dimension at least 1. Hence
$$\sigma_a\cdot\sigma_{n-s,n-k+2,...,n-1,n}\neq0.$$
Let $\mathbb{P}(G_{n-s-1})=\mathbb{P}^{n-s-2}$ be another general projective linear space of dimension $n-s-2$. Since $\dim(Z)=s$, $Z$ does not meet $\mathbb{P}(G_{n-s-1})$ and hence 
$$\sigma_a\cdot\sigma_{n-s-1,n-k+2,...,n-1,n}=0.$$
By Pieri's formula, this can happen only when $s=a_k-1$:
$$\sigma_a\cdot\sigma_{n-a_k+1,n-k+2,...,n-1,n}=\sigma_{1,a_1+1,...,a_{k-1}+1}$$
and for all $c\leq n-a_k$,
$$\sigma_a\cdot\sigma_{c,n-k+2,...,n-1,n}=0.$$

Now suppose $d\geq2$. Since $\dim(Z\cap \mathbb{P}(G_{n-s}))=0$, any line joining two of the $d$ points $p_1,...,p_d$ in the intersection cannot be contained in $Z$, and therefore every projective $k-1$ linear space parametrized by $X$ can contain at most one of $p_1,...,p_d$. This implies the intersection of $X$ and $\Sigma_{n-a_k+1,n-k+2,...,n-1,n}$ has at least $d$ irreducible components. On the other hand, by Pieri's formula,
$$[X\cap\Sigma_{n-a_k+1,n-k+2,...,n-1,n}]=\sigma_{1,a_1,...,a_{k-1}}$$
is a Schubert class and since a Schubert class is indecomposable, $d=1$.
\end{proof}

The natural isomorphism $V\cong V^*$ induces an isomorphism $G(k,n)\cong G(n-k,n)$. Under this duality, the Schubert class $\sigma_a\subset G(k,n)$ is taken to the Schubert class $\sigma_b\in G(n-k,n)$, where $b$ can be found by taking the transpose of the associated Young diagram to $a$ \cite[Ex 4.31]{3264}. Suppose $a_i=i$ and $a_{i+1}\neq i+1$. Then under the duality $G(k,n)\cong G(n-k,n)$, the Schubert class $\sigma_a\in A(G(k,n))$ is taken to the Schubert class $\sigma_b\in A(G(n-k,n))$, where $b_k=n-i$. The above proposition gives the following:
\begin{proposition}\cite[Proposition 3.1]{Coskun2011RigidAN}
\label{containing point}
 If $a_i=i$ and $a_{i+1}\neq i+1$, then there exists a fixed $i$-dimensional space $F_{i}$ such that $F_i\subset \Lambda$ for all $\Lambda\in X$.
\end{proposition}

The proof of \ref{contained in plane} also gives the following result.
\begin{corollary}
\label{2.3}
Let $X$ be a subvariety of $G$ with $[X]=\sigma_{a_1,...,a_k}$. The k-planes parametrized by $X$ sweep out a projective linear space $\mathbb{P}(F_{a_k})$. Let $p\in \mathbb{P}(F_{a_k})$ be a general point and define $[X_p]:=\{\Lambda\in X|p\in\mathbb{P}(\Lambda)\}$. Then $$[X_p]=\sigma_{1,a_1+1,...,a_{k-1}+1}.$$
Let $\Lambda\in X$ be a general point in $X$ and $H$ be a general hyperplane containing $\Lambda$. Define $X_H:=\{\Lambda\in X|\Lambda\subset H\}$. Suppose $a_s=s$ and $a_{s+1}\neq s+1$. Then $$[X_H]=\sigma_{a_1,...,a_s,a_{s+1}-1,...,a_k-1}.$$
\end{corollary}
The second part follows by the duality $G(k,n)\cong G(n-k,n)$.

Now we prove Theorem \ref{index} by induction.
\begin{proof}[Proof of Theorem \ref{index}]
$(1)$ and $(2)$ come from Proposition \ref{contained in plane} and Proposition \ref{containing point}.

For $(3)$, assume that $a_i\leq a_{i+1}-3$. If $a_k\neq n$, then by Proposition \ref{contained in plane}, $X$ is contained in a sub-Grassmannian. Hence we can reduce to the case where $a_k=n$. We will use induction on $k,n$ and the sequence $(a_i)$, and use the ordering $(a_1,...,a_k)<(a_1',...,a_j')$ if $k<j$ or if $k=j$ and $a_i=a'_i$ for $1\leq i<s$ and  $a_s<a'_s$.

The theorem is trivial when $n\leq 3$ or $k=1$. When $a=(1,2,...,k)$, it reduces to Proposition \ref{containing point}.

Now assume the statement is true for all $a'<a$.
\begin{itemize}
\item If $a_1=1$, then by Proposition \ref{containing point}, there exists $F_1$ such that $F_1\subset\Lambda$ for all $\Lambda\in X$. Suppose $i\geq 2$. Let $\bar{\Lambda}$ be the image of $\Lambda$ under the projection from $F_1$, and let $\bar{X}$ be the collection of $\bar{\Lambda}$. Then $[\bar{X}]=\sigma_{a'}\in A(G(k-1,n-1))$ where $a'=(a_2-1,...,a_k-1)$. By assumption, $a'_{i}=a_i-1\leq a_{i+1}-4=a'_{i+1}-3$. Since $a'<a$, by induction there is a fixed $\bar{F}_{a_i-1}$ such that $\dim(\bar{\Lambda}\cap \bar{F}_{a_i-1})\geq i-1$ for all $\bar{\Lambda}\in \bar{X}$. Let $F_{a_i}$ be the pre-image of $\bar{F}_{a_{i-1}}$. Then $\dim(\Lambda\cap F_{a_i})\geq1+\dim(\bar{\Lambda}\cap \bar{F}_{a_i-1})\geq i$. 

For the uniqueness, say if $\dim(G_{a_i}\cap \Lambda)\geq i$ for all $\Lambda\in X$, and let $\bar{G}$ be the image of $G_{a_i}$ under the projection from $F_1$. If $F_1$ is contained in $G_{a_i}$, then by induction $\bar{F}_{a_i-1}=\bar{G}$ and hence $G_{a_i}=F_{a_i}$. If  $F_1$ is not contained in $G_{a_i}$, then $\dim(\bar{G})=\dim(G_{a_i})$ and $\dim(\bar{G}\cap \bar{\Lambda})\geq i$ for all $\bar{\Lambda}\in\bar{X}$. Then for a general codimension 1 subspace $\bar{G}'$ of $\bar{G}$, $\dim(\bar{G}'\cap \bar{\Lambda})\geq i-1$ for all $\bar{\Lambda}$, which contradicts the uniqueness hypothesis in the induction.

\item If $a_1\neq 1$, then for a general hyperplane $H$, by Lemma \ref{2.3}, $[X_H]=\sigma_{a_1-1,...,a_k-1}$. By induction, there exist a unique $(a_i-1)$-dimensional linear space $F^H_{a_i-1}$ such that $\dim(\Lambda\cap F^H_{a_i-1})\geq i$ for all $\Lambda\in X_H$. As we vary $H$, let $Z$ be the projective variety swept out by $\mathbb{P}(F^H_{a_i-1})$. Clearly $\dim(Z)\geq a_i-2$. For a general $H'$ that does not contain $F^H_{a_i-1}$, $F^{H'}_{a_i-1}\neq  F^H_{a_i-1}$, and hence $\dim(Z)\geq a_1-1$. We claim that $\dim(Z)=a_i-1$. 

Suppose, for a contradiction, that $\dim{Z}=a\geq a_i$. Let $G_\bullet$ be a general complete flag. Then $\mathbb{P}(G_{n-a_i})$ will meet $Z$ in finitely many points. By the construction of $Z$, there exists a hyperplane $H$ such that $$\dim(G_{n-a_i}\cap F^H_{a_i-1})=1.$$ Since $G_{n+2-a_i}$ is a general linear space containing $G_{n-a_i}$, by dimension reason we may assume
$$\dim(G_{n-a_i}\cap F^H_{a_i-1})=\dim(G_{n+2-a_i}\cap F^H_{a_i-1})=1.$$ 
Let $\Sigma$ be the Schubert variety defined by the partial flag:
$$G_{n+2-a_k}\subset...\subset G_{n+2-a_i}\subset...\subset...\subset G_{n+2-a_1}.$$ 
Notice that $[\Sigma]\cdot\sigma_{a_1-1,...,a_k-1}=\sigma_{(n-k)^k}$. Let $\Lambda\in X_H\cap\Sigma$. Then $\dim(\Lambda\cap F^H_{a_i-1})=i$, $\dim(\Lambda\cap G_{n+2-a_{i+1}})=k-i$ and $\dim(\Lambda\cap G_{n+2-a_{i}})=k-i+1$, and therefore $$\dim(\Lambda\cap F^H_{a_i-1}\cap G_{n+2-a_i})=1.$$
Since $a_{i+1}\geq a_i+3$, $n-a_i>n-a_{i+1}+2$ and thus $G_{n+2-a_{i+1}}\cap F^H_{a_i-1}=0$, 
$$\dim(\Lambda\cap G_{n-a_i})\geq \dim(\Lambda\cap F^H_{a_i-1}\cap G_{n+2-a_i})+\dim(\Lambda\cap G_{n+2-a_{i+1}})=k-i+1.$$
Hence we proved that for a general $(n-a_i)$-dimensional vector space $G_{n-a_i}$, there exists a $\Lambda\in X$ that meets $G_{n-a_i}$ in a $(k-i+1)$-dimensional subspace, which is a contradiction since
$$\sigma_a\cdot \sigma_b=0,$$
where $b_j=n-a_i+j-(k-i+1)$ for $1\leq j\leq k-i+1$ and $b_j=n+j-k$ for $k-i+2\leq j\leq k$. We conclude that $\dim(Z)=a_i-1$. 

Now we claim that $Z$ is linear. This implies $Z=\mathbb{P}(F_{a_i})$ and for every $\Lambda\in Y$, $\dim(\Lambda\cap F_{a_i})\geq i$. If $\dim(Z)=a_i-1\geq 2$, then by Bertini's theorem, a general hyperplane section of $Z$ is irreducible. Since $\mathbb{P}(F_{a_{i-1}}^H)\subset Z\cap H$, we must have $\mathbb{P}(F_{a_{i-1}}^H)=Z\cap H$, and therefore $Z$ is linear. If $a_i-1=1$, it means $i=1$ and $a_1=2$. Suppose for a contradiction that $\deg(Z)=d\geq 2$. Let $H$ be a general hyperplane, then $H$ will meet $Z$ in $d$ projective points $p_1,...,p_d$. Assume that $p_1=\mathbb{P}(F_2^H)$ is the point defined by $H$. Let $H'\neq H$ be another hyperplane that defines $p_2=\mathbb{P}(F_2^{H'})$ and let $X_{H\cap H'}$ be the locus of $k$-planes parametrized by $X$ that are contained in $H\cap H'$. By Corollary \ref{2.3},
$$[X_{H\cap H'}]=\sigma_{1,a_2-2,...,a_k-2}.$$
By assumption, $a_2-a_1=a_2-2\geq 3$ and therefore by Proposition \ref{containing point}, there is a unique projective point $p$ which is contained in every $k$-planes parametrized by $X_{H\cap H'}$. We reach a contradiction since both $p_1$ and $p_2$ are contained in every $k$-plane parametrized by $X_{H\cap H'}$. This proves the claim. 

For the uniqueness, let $F'_{a_i}$ be another $a_i$-dimensional vector space such that $\dim(F'_{a_i}\cap \Lambda)\geq i$ for all $\Lambda$ in $X$. Then for a general hyperplane $H$, $F'_{a_i}\cap H=F^H_{a_i-1}=F_{a_i}\cap H$. Hence $F'_{a_i}=F_{a_i}$.
\end{itemize}

$(4)$ can be obtained by duality. Suppose $a_{i}=a_i+1$ and $a_{i+1}\neq a_i+1$. Under the duality $G(k,n)\cong G(n-k,n)$, $\sigma_a$ is taken to $\sigma_b\in A(G(n-k,n))$, where $b_{n-k+i-a_i+1}-b_{n-k+i-a_i}\geq 3$. Apply $(3)$ to $X^*$ which consists of dual linear spaces in $X$, there exists a fixed $W_{n-a_i}$ such that $\dim(\Lambda^*\cap W_{n-a_i})\geq n-k+i-a_i$ for all $\Lambda^*\in X^*$. Equivalently, $\dim(\Lambda\cap W^*_{n-a_i})\geq i$ for all $\Lambda\in X$. In this case $F_{a_i}=W^*$.
\end{proof}
\begin{remark}
The converse of Theorem \ref{index} is also true. We refer the reader to \cite[Theorem 1.3]{Coskun2011RigidAN} for a construction of a counterexample when all conditions fail.
\end{remark}

As a corollary, we recover the classification of rigid Schubert classes in Grassmannians which was first proved by Coskun \cite{Coskun2011RigidAN}.

\begin{theorem}\cite[Theorem 1.3]{Coskun2011RigidAN}
\label{rigid class in g}
The Schubert class $\sigma_a=\sigma_{a_1,...,a_k}$ is rigid if and only if all essential subindices are rigid.
\end{theorem}

The proof will be based on Theorem \ref{index} and the following lemma:
\begin{lemma}\label{lemma 2.3}
Assume $a_i,a_j$ are rigid, $i<j$ and let $F_{a_i},F_{a_j}$ be defined as in Theorem \ref{index}. Then $F_{a_i}\subset F_{a_j}$.
\end{lemma}

\begin{proof}
First assume $a_i=1$. If $F_1\not\subset F_{a_j}$, let $F_{a_j+1}=F_1+F_{a_j}$. Then for every $\Lambda\in X$, $$\dim(\Lambda\cap F_{a_j+1})=\dim(\Lambda\cap F_{1})+\dim(\Lambda\cap F_{a_j})=j+1$$ and hence for every codimension 1 linear subspace $V$ in $F_{a_j+1}$,
$$\dim(\Lambda\cap V)=j.$$
Therefore $a_j$ is not essential (since such $F_{a_j}$ is not unique), we reach a contradiction. 

Now assume $a_i\neq1$. We will prove it using induction on $n$. It is trivial when $n=1$. 

If $a_1=1$, then $F_{1}\subset F_{a_j}$ for all essential rigid $j>1$. Let $\bar{X}\subset G(k-1,n-1)$, $\bar{F}_{a_i}$ and $\bar{F}_{a_j}$ be the image of $X$, $F_{a_i}$ and $F_{a_j}$ under the projection from $F_1$. By induction $\bar{F}_{a_i}\subset\bar{F}_{a_j}$ and hence $F_{a_i}\subset F_{a_j}$.

If $a_1\neq1$, then for every hyperplane $H$, $X_H\subset G(k,n-1)$ is non-empty and by induction $$F_{a_i}\cap H=F_{a_i}^H\subset F^H_{a_j}=F_{a_j}\cap H.$$ Hence $F_{a_i}\subset F_{a_j}$.
\end{proof}

\begin{proof}[Proof of Theorem \ref{rigid class in g}]
If one of the essential indices is not rigid, then by Theorem \ref{index}, we can find a subvariety $X$ which is not a Schubert variety that presents $\sigma_a$.

Now assume all essential indices are rigid. Let $I$ be the index set consisting of all essential indices in $a$, and $X$ be a subvariety representing $\sigma_a$. Then by Theorem \ref{index}, for each $i\in I$, we can find a fixed linear space $F_{a_i}$ such that $\dim(\Lambda\cap F_{a_i})\geq i$ for all $\Lambda\in X$. Moreover, by Lemma \ref{lemma 2.3}, $\{F_{a_i}\}$ forms a partial flag. Define
$$\Sigma:=\{\Lambda\in G|\dim(\Lambda\cap a_i)\geq i, i\in I\}.$$
Then $\Sigma$ is a Schubert variety and $X\subset \Sigma$. Since $\dim(X)=\dim(\Sigma)$, $X$ is the Schubert variety $\Sigma$. 
\end{proof}

\section{Restriction Varieties}\label{sec2}
In the next section, we will make frequent use of restriction varieties. For the reader's convenience, we review some basic facts about restriction varieties and recall an algorithm due to Coskun \cite{Coskun2011RestrictionVA} computing their cohomology classes in terms of Schubert classes.

Let $V$ be an $n$-dimensional vector space over $\mathbb{C}$ and let $q$ be a nonsingular symmetric bilinear form. 
Geometrically, the form $q$ defines a smooth quadric hypersurface $Q$ in $\mathbb{P}(V)$ by setting $Q(\bar{x})=q(x,x)$. A subspace $W$ is isotropic if and only if $\mathbb{P}(W)$ lies on $Q$. 

Let $Q_{d}^r$ denote a subquadric of $Q$ of corank $r$ which is obtained by restricting $Q$ to a $d$-dimensional linear space. Notice that for an isotropic subspace $F_{a_i}$, the intersection $\mathbb{P}(F_{a_i}^\perp)\cap Q$ is a sub-quadric $Q_{n-a_i}^{a_i}$. Therefore we may re-define the Schubert varieties directly with respect to a flag of isotropic subspaces and sub-quadrics:
$$F_{a_1}\subset...\subset F_{a_s}\subset Q_{d_{k-s}}^{r_{k-s}}\subset...\subset Q_{d_1}^{r_1}$$
where $d_j+r_j=n$, $a_i\neq r_j+1$ for all $i,j$, and by requiring the singular locus of $Q_{d_j}^{r_j}$ is contained in the singular locus of $Q_{d_{j+1}}^{r_{j+1}}$ and $F_{a_i}$ is either contained or contains the singular locus of $Q_{d_j}^{r_j}$. From here, we can see under the natural inclusion $$OG(k,n)\hookrightarrow OG(k,n+1),$$ the image of a Schubert variety is no longer a Schubert variety in $OG(k,n+1)$, since for the Schubert varieties in $OG(k,n+1)$, the sum of $r_i$ and $d_i$ should be $n+1$ for the defining quadrics. Due to this observation, we extend our focus to the restriction varieties which allow the coranks of the defining quadrics to be less than $n-d_i$. 

\begin{definition}\cite[Definition 4.2]{Coskun2011RestrictionVA}
\label{sequence}
Given a sequence consisting of isotropic subspaces $F_{a_i}$ of $V$ and subquadrics $Q_{d_j}^{r_j}$ :
$$F_{a_1}\subsetneq...\subsetneq F_{a_s}\subsetneq Q_{d_{k-s}}^{r_{k-s}}\subsetneq...\subsetneq Q^{r_1}_{d_1}$$ such that

$(1)$ For every $1\leq j\leq k-s-1$, the singular locus of $Q_{d_j}^{r_j}$ is contained in the singular locus of $Q_{d_{j+1}}^{r_{j+1}}$;

$(2)$ For every pair $(F_{a_i},Q_{d_j}^{r_j})$, $\dim(F_{a_i}\cap \text{Sing}(Q_{d_j}^{r_j}))=\min\{a_i,r_j\}$;

$(3)$ Either $r_i=r_1=n_{r_1}$ or $r_t-r_i\geq t-i-1$ for every $t>i$. Moreover, if $r_t=r_{t-1}>r_1$ for some $t$, then $d_i-d_{i+1}=r_{i+1}-r_i$ for every $i\geq t$ and $d_{t-1}-d_t=1$;

$(A1)$ $r_{k-s}\leq d_{k-s}-3$;

$(A2)$ $a_i-r_j\neq 1$ for all $1\leq i\leq s$ and $1\leq j\leq k-s$;

$(A3)$ Let $x_j=\#\{i|a_i\leq r_j\}$. For every $1\leq j\leq k-s$, $$x_j\geq k-j+1-\left[\frac{d_j-r_j}{2}\right].$$
We define the associated restriction variety 
$$V(F_\bullet,Q_\bullet):=\{\Lambda\in OG(k,n)|\dim(\Lambda\cap F_{a_i})\geq i,\dim(\Lambda\cap Q_{d_j}^{r_j})\geq k-j+1,1\leq i\leq s, 1\leq j\leq k-s\}.$$
\end{definition}
\begin{remark}
In case $n$ is even and $n=2n_i$, we denote $F_{n_i}$ and $F_{n_i}'$ the isotropic subspaces in different connected components.
\end{remark}
\begin{remark}
The Schubert varieties in $OG(k,n)$ are the restriction varieties with $d_j+r_j=n$ for all $1\leq j\leq k-s$.
\end{remark}

\begin{definition}\cite[Definition 4.4]{Coskun2011RestrictionVA}
Given a sequence $(F_\bullet,Q_\bullet)$ that satisfies the conditions $(1)-(3)$ as in the Definition \ref{sequence}, the associated quadric diagram $D$ consists of

$(I)$: a sequence of $n$ numbers where $l$-th number equals to $j$ if $r_{j-1}< l\leq r_{j}$ (we set $r_0=0$) and equals to $0$ if $l>r_{k-s}$; and

$(II)$: for each $i$, a bracket $]$ right after $n_i$-th digit and a brace $\}$ right after $d_j$-th digit. In case $n$ is even and the sequence contains $F'_{n/2}$, use $]'$ instead of $]$ after $(n/2)$-th digit.

The associated diagram $D$ is called admissible if the sequence $(F_\bullet,Q_\bullet)$ also satisfies the conditions $(A1)-(A3)$ as in the Definition \ref{sequence}.
\end{definition}

\begin{definition}\label{da}\cite[Definition 3.11]{Coskun2014RigidityOS}
Let $D$ be an admissible diagram. If $d_j+r_j<d_{j-1}+r_{j-1}$ for some $j$, set
$$\kappa:=\max\{j|d_j+r_j<d_{j-1}+r_{j-1}\}.$$
If $d_j+r_j=d_1+r_1\neq n$ for all $j$, set $\kappa=1$.

Let $D^a$ be the diagram obtained by changing the $(r_\kappa+1)$-st digit in $D$ to $\kappa$. If there is a bracket in $D^a$ to the right of the $(r_\kappa+1)$-st digit, let $D^b$ be the diagram obtained from $D^a$ by moving the leftmost bracket among such brackets to the right of the $(r_\kappa+1)$-st digit.
\end{definition}
\begin{definition}\cite[Definition 4.6]{Coskun2011RestrictionVA}
Assume $n_s>r_\kappa$. If $r_j\geq n_{x_\kappa+1}$ for some $j$, set $$y_\kappa=\max\{j|r_j\leq n_{x_\kappa+1}\}.$$ Otherwise, set $y_k=k-s+1$.
\end{definition}

\begin{algorithm}
\caption{\cite[Algorithm 3.8]{Coskun2011RestrictionVA} Diagrams derived from $D^a$}.

\begin{steps}
\item If $D^a$ fails condition $(A3)$ in the Definition \ref{sequence}, then discard $D^a$. Else, proceed to the next step.
\item If $D^a$ fails condition $(A2)$, then change the digit to the right of the rightmost $\kappa$ to $\kappa$ and move the $\kappa$-th brace from the right one position to the left. Else, proceed to the next step.
\item If $D^a$ fails condition $(A1)$, then replace $D^a$ by two identical diagrams $D^{a_1}$ and $D^{a_2}$ obtained by replacing the leftmost brace with a bracket one position to the left and changing the digits equal to $k-s$ to $0$. If the rightmost bracket in $D^{a_2}$ is right after the $\frac{n}{2}$-th digit, replace this bracket by $]'$. 

\end{steps}

\end{algorithm}

\begin{algorithm}
\caption{\cite[Algorithm 3.9]{Coskun2011RestrictionVA} Diagrams derived from $D^b$}
\begin{algorithmic}
\STATE If $D^b$ fails condition $(A2)$, suppose it fails for the $j$-th bracket. Let $i$ be the integer immediately to the left of the $j$-th bracket. Replace this $i$ by $i-1$ and move $i-1$-st brace from the right one position to the left. Repeat either until the resulting sequence is admissible or two braces occupy the same position. In the latter case, discard $D^b$.
\end{algorithmic}
\end{algorithm}

\begin{algorithm}
\caption{\cite[Algorithm 3.10]{Coskun2011RestrictionVA} Diagrams derived from an admissible quadric diagram $D$}\label{algorithm2}
\begin{algorithmic}
\REQUIRE an admissible diagram $D$. 
\IF{ $r_j+d_j=n$ for every $1\leq j\leq k-s$, }
\RETURN $D$
\ELSE
\IF{$n_{x_\kappa+1}-r_\kappa-1>y_\kappa-\kappa$ or $n_s\leq r_\kappa$ in $D$, }
\RETURN the diagrams derived from $D^a$;
\ENDIF
\IF{$D^a$ violates condition (A3) in Definition \ref{sequence},}
\RETURN the diagrams derived from $D^{b}$;

\ELSE 
\RETURN the diagrams derived from both $D^a$ and $D^b$.
\ENDIF
\ENDIF
\end{algorithmic}
\end{algorithm}
\newpage
\begin{theorem}\cite[Theorem 5.12]{Coskun2011RestrictionVA}
Let $V(F_\bullet,Q_\bullet)$ be a restriction variety. Then the rational equivalence class of $V(F_\bullet,Q_\bullet)$ is the sum of classes of restriction varieties derived from $V(F_\bullet,Q_\bullet)$ by Algorithm \ref{algorithm2}.
\end{theorem}

\begin{remark}\label{classing}
The above algorithm can also be used to find the rational equivalence class of the image of a Schubert variety under the natural inclusion 
$$i: OG(k,n)\hookrightarrow G(k,n).$$
Let $X=\Sigma_{a;b}$ be a Schubert variety in $OG(k,n)$. By passing to $OG(k,2n+1)$, $X$ can be considered as a restriction variety defined by
$$F_{a_1}\subset...\subset F_{a_s}\subset Q_{n-b_{k-s}}^{r_{k-s}}\subset...\subset Q_{n-b_1}^{b_1}.$$
Applying Algorithm \ref{algorithm2} to the quadric diagram corresponding to this restriction variety, we end up with quadric diagrams without braces. Those quadric diagrams give the cohomology class of $i(X)$ in terms of Schubert classes in $G(k,n)$.

\end{remark}

\section{Rigidity problems in Orthogonal Grassmannians}\label{sec3}
In this section, we prove Theorem \ref{main theorem}. We begin with an example in $OG(2,6)$.
\begin{example}\label{example1}
Consider the orthogonal Grassmannian $OG(2,6)$. Let $(a;b)$ be a Schubert index with $s=1$. We claim that if $a=1$ is essential, then it is rigid. The possible Schubert classes in $OG(2,6)$ with $a=1$ essential are $$\sigma_{1,3;}\ \  ,\ \ \ \sigma_{1;2}\ \ ,\ \ \  \sigma_{1;1}$$ We will investigate them separately. 

For the Schubert classes $\sigma_{1,3;}$ and $\sigma_{1;2}$, consider the natural inclusion $$i:OG(2,6))\hookrightarrow G(2,6).$$ Let $X$ be a subvariety representing $\sigma_{1,3}$ or $\sigma_1^2$ in $OG(2,6)$, then $$[i(X)]=\sigma_{1,3;}\in A(G(2,6)).$$ By Theorem \ref{index}, there is a unique one dimensional vector space $F_1$ that is contained in every $k$-plane parametrized by $X$. Therefore the sub-index $a_1=1$ is rigid. 

Now consider the Schubert class $\sigma_{1;1}$. Let $X$ be a subvariety representing $\sigma_{1;1}$ in $OG(2,6)$. The $k$-planes parametrized by $X$ sweep out a quadric $Q_X$ of dimension $3$ with corank at most $1$ \cite[Lemma 6.2]{Coskun2014RigidityOS}. We will show that the corank of $Q_X$ has to be 1 and the singular locus of $Q_X$ is the unique projective point that is contained in every $k$-plane parametrized by $X$.

Suppose, for a contradiction, that $Q_X$ is smooth, then $X$ can be viewed as a subvariety in $OG(2,5)$. The only Schubert class in $OG(2,5)$ of the same dimension as $\sigma_{1;1}\in A(OG(2,6))$ is $\sigma_{2;0}$. Let $\Sigma=\Sigma_{2;0}$ be a Schubert vairety in $OG(2,5)$, then the image of $\Sigma$ under the natural inclusion $i':OG(2,5)\hookrightarrow OG(2,6)$ is the restriction variety defined by $$F_2\subset Q_{5}^0.$$ Apply Algorithm \ref{algorithm2}, we get$$[X]=[i'(\Sigma)]=\sigma_{1;1}+2\sigma_{2,3;}\neq \sigma_{1;1}\in A(OG(2,6)),$$ which is a contradiction. Therefore $Q_X$ has corank $1$. 

Let $q$ be the singular point of $Q_X$. Consider the following incidence correspondence
$$I:=\{(L,F_3)|L\subset\mathbb{P}(F_3), L\in X, F_3\subset Q_X\}\subset X\times OG(3,6).$$
Let $\pi_1:I\rightarrow X$ and $\pi_2:I\rightarrow OG(3,6)$ be the two projections. Let $Y$ be the image of $I$ under the second projection $\pi_2$. Let $F$ be a general point in $Y$. We claim that 
$$[\pi_1\circ\pi_2^{-1}(F)]=\sigma_{1,3}\in A(G(2,F)).$$
For a line $L\in X$, $L^\perp\cap Q$ is a union $\mathbb{P}(F_3)\cup \mathbb{P}(F'_3)$ of two $3$-planes belonging to different component. Therefore the fibers of $\pi_1$ have dimension $0$ and hence
$\dim(I)=\dim(X)=2.$
Since $Q_X$ is swept out by the lines parametrized by $X$, the image of $\pi_2$ has dimension at least $1$. Meanwhile the locus of isotropic subspaces of dimension 3 contained in $Q_X$ has dimension $1$, we obtain $\dim(\pi_2(I))=1$ and thus a general fiber of $\pi_2$ over the image has dimension $1$. Therefore$$[\pi_1\circ\pi_2^{-1}(F)]=m\sigma_{1,3}\in A(G(2,F)).$$ 
Notice that for a general point $p\in Q_X$, there is a unique line parametrized by $X$ that contains $p$. Therefore $m=1$.

By Theorem \ref{index}, there exists a unique point $p_F$ that is contained in every line parametrized by $\pi_1\circ\pi^{-1}_2(F)$. We claim that $p_F$ is the singular point $q$, and therefore every line parametrized by $X$ is contained in some $\pi_1\circ\pi^{-1}_2(F)$ and thus contains $q$.

Suppose for a contradiction that $p_F\neq q$. Let $p\neq q$ be a general point in $F$ that is not contained in the line $\overline{qp_F}$. Let $F'$ be the other $3$-plane containing $\overline{pq}$. Then there are at least two different lines parametrized by $X$ passing through $p$, namely $\overline{pp_F}$ and $\overline{pp_{F'}}$. This contradicts the fact that there is only one line parametrized by $X$ that passes through a general point in $Q_X$. We conclude that $p_F=q$ and it completes the proof. 
\end{example}

More generally, we have:

\begin{proposition}
Assume $n\geq 2k+2$. Let $\sigma_{a;b}$ be a Schubert class in $OG(k,n)$. If $s=k-1$ and $a=(1,...,k-1)$,  then $a_{k-1}=k-1$ is rigid.\label{1 is rigid}
\end{proposition}
We will need the following lemmas:
\begin{lemma}\label{5.5}
Assume $n\geq 2k+2$. Let $X$ be a subvariety of $OG(k,n)$ representing the Schubert class $\sigma_{a;b}$ where $a=(1,...,k-1)$. Let $Y$ be an irreducible subvariety of $Q$ with the smallest dimension such that $\dim(Y\cap\mathbb{P}(\Lambda))\geq k-2$ for all $\Lambda\in X$. Then $\dim(Y)=k-2$ or $n-b-3$.
\end{lemma}
\begin{proof}
Let $X$ be a representative of $\sigma_{a;b}$ in $OG(k,n)$ and assume $\dim(Y)>k-2$. Let $p\in Y$ be a general point and define $$X_p:=\{\Lambda\in X|p\in\mathbb{P}(\Lambda)\}.$$ 
We claim that $$[X_p]\cdot \sigma_{;k,k-2,...,0}=0$$
and $$\dim(Y)=n-b-3.$$
Suppose, for a contradiction, that $[X_p]\cdot \sigma_{;k,k-2,...,0}\neq0.$
Let $F_{k-1}$ be a general isotropic subspace of dimension $k-1$. Since $\dim(Y)>k-2$, $$Y\cap \mathbb{P}(F_{k-1}^\perp)\neq\emptyset.$$ 
Let $p\in Y\cap \mathbb{P}(F_{k-1}^\perp)$, $p\notin F_{k-1}$ and let $p^\perp$ be the orthogonal complement of $p$. Then $$F_{k-1}^\perp\not\subset p^\perp.$$ Let $q\in F_{k-1}^\perp$ be a general point which is not contained in $p^\perp$, and let $F_{k}$ be the span of $F_{k-1}$ and $q$. Then the line $\overline{pq}$ is not isotropic and hence $p\notin F_{k}^\perp$. Since $[X_p]\cdot \sigma_{;k,k-2,...,0}\neq0$, there exists a $k$-plane $\Lambda$ parametrized by $X$ such that $p\in\mathbb{P}(\Lambda)$ and $\mathbb{P}(\Lambda)$ meets $\mathbb{P}(F_{k}^\perp)$ in a different point $p'\neq p$. Since both $p$ and $p'$ are contained in $F_{k-1}^\perp$, $\dim(\Lambda\cap F_{k-1}^\perp)\geq 2$. This contradicts the fact that $[X]\cdot\sigma_{;k,k-1,k-3,...,0}=0$. We conclude that $[X_p]\cdot \sigma_{;k,k-2,...,0}=0$. 

Now consider the incidence correspondence 
$$I=\{(p,\Lambda)|p\in Y,\Lambda\in X, p\in\mathbb{P}(\Lambda)\}.$$
Let $\pi_1:I\rightarrow Y$ and $\pi_2:I\rightarrow X$ be the two projections. Since $\sigma_{;k,k-2,...,0}$ has codimension $1$ in $A(OG(k,n))$ and $[X_p]\cdot \sigma_{;k,k-2,...,0}=0$, we obtain $\dim(X_p)=0$. Therefore a general fiber of $\pi_1$ has dimension $0$. By the semi-continuity, for a general $\Lambda\in X$, $$\dim(\mathbb{P}(\Lambda)\cap Y)=k-2,$$ and hence a general fiber of $\pi_2$ has dimension $k-2$. Under the natural inclusion $OG(k,n)\hookrightarrow G(k,n)$, by Remark \ref{classing} and Algorithm \ref{algorithm2}, we get $$[i(X)]=2\sigma_{1,2,...,k-1,n-b-1}\in A(G(k,n))$$ and hence $\dim(X)=\dim(i(X))=n-b-k-1$. Therefore
$$\dim(Y)=\dim(X)+(k-2)=n-b-3.$$ 
\end{proof}
\begin{lemma}\label{lemma 5.7}
Let $\sigma_{a;b}$ be a Schubert class in $OG(k,n)$ with $a=(1,...,k-1)$ and $b=\frac{n}{2}-2$. Let $X$ be a subvariety representing the class $\sigma_{a;b}$. Let $Q_X$ be the quadric swept out by the $k$-planes parametrized by $X$. For a general maximal isotropic subspace $F=F_{b+2}$ whose projectivization is contained in $Q_X$ and contains at least one $k$-plane parametrized by $X$, define
$$X_{F}:=\{\Lambda\in X|\Lambda\subset \mathbb{P}(F)\}.$$
Then the cohomology class of $X_F$ in $G(k,F)$ is given by
$$[X_F]=\sigma_{1,2,...,k-1,b+2}.$$
\end{lemma}
\begin{proof}
Consider the incidence correspondence
$$I:=\{(\Lambda,F_{b+2})|\Lambda\subset \mathbb{P}(F_{b+2})\subset Q_X, \Lambda\in X\}\subset X\times OG(b+2,n),$$
and let $\pi_1:I\rightarrow X$ and $\pi_2:I\rightarrow OG(b+2,n)$ be the two projections. Let $p\in \Lambda$ be a point that is not contained in the singular locus of $Q_X$. Then the orthogonal complement $p^\perp$ will cut out $Q_X$ into a union of $\mathbb{P}(F_{b+2})\cup \mathbb{P}(F'_{b+2})$ (belong to different component). Therefore a general fiber of $\pi_1$ has dimension $0$, and hence $$\dim(I)=\dim(X)=n-b-k-1.$$ 
Since the $k$-planes parametrized by $X$ sweep out the quadric $Q_X$, $\dim(\pi_2(I))\geq1$. Since the locus of all maximal isotropic subspaces contained in $Q_X$ has dimension $1$, we conclude that $$\dim(\pi_2(I))=1.$$ Therefore a general fiber of $\pi_2$ over the image has dimension $n-b-k-2$. $$\dim(X_F)=n-b-k-2=b-k+2.$$ By counting the dimension we must have $$[X_F]=m\sigma_{1,...,k-1,b+2}\in A(G(k,b+2)).$$ Since through a general point $p$ in $Q$, there is only one $k$-plane parametrized by $X$ (as the class $[X_p]=\sigma_{1,...,k}$ is the class of a point), we get $m=1$. Hence $$[X_F]=\sigma_{1,...,k-1,b+2}$$
\end{proof}

\begin{lemma}\label{lemma 5.8}
Let $\sigma_{a;b}$ be a Schubert class in $OG(k,n)$ with $a=(1,...,k-1)$. Let $X$ be a subvariety representing the class $\sigma_{a;b}$. Let $Q_X$ be the quadric swept out by the $k$-planes parametrized by $X$. Then the corank $r$ of the quadric $Q_X$ is at least $k-1$.
\begin{proof}
Notice that $b\geq k-1$ by the definition of a Schubert index (see Definition \ref{Schubert variety}). 

Suppose for a contradiction that $r<k-1$. Then $Q_X$ is contained in a smooth quadric of dimension $n-b+r$ and therefore $X$ can be viewed as a subvariety of $OG(k,n-b+r)$. Let $i:OG(k,n-b+r)\hookrightarrow OG(k,n)$ be the natural inclusion. Since $[i(X)]=\sigma_{1,...,k-1;b}$, $[X]=\sigma_{a';b'}\in A(OG(k,n-b+r))$ is a Schubert class. Since $Q_X$ has dimension $$n-b-2=n-b+r-b'-2,$$ we obtain $b'=r$. Compare the dimension of $\sigma_{a';b'}$ and $\sigma_{a;b}$, we must have $$a'=(1,...,r,r+2,...,k).$$ Apply Algorithm \ref{algorithm2}, we get $$[i(\Sigma_{a';r})]=\sigma_{a;b}+2\sigma_{1,...,k-1,b+1;}\neq \sigma_{a;b},$$ which is a contradiction.
\end{proof}

\end{lemma}

\begin{proof}[Proof of Proposition \ref{1 is rigid}]
Let $X$ be a subvariety of $OG(k,n)$ representing the Schubert class $\sigma_{a;b}$. The $k$-planes parameterized by $X$ sweep out a quadric $Q_X$ of dimension $n-b-2$ \cite[Lemma 6.2]{Coskun2014RigidityOS}. We will show that there exists a unique isotropic subspace $F_{k-1}$ such that $F_{k-1}\subset\Lambda$ for all $\Lambda\in X$.

The proof will be done by induction on $n$ and $b$ using the ordering $(n,b)<(n',b')$ if $n<n'$ or $n=n'$ and $b>b'$.

\begin{itemize}
\item If $n=2k+2$ and $b=k$, let $i:OG(k,n)\hookrightarrow G(k,n)$ be the natural inclusion. Then
$$[i(X)]=\sigma_{1,...,k-1,k+1}\in A(G(k,n)).$$
By Theorem \ref{index}, the sub-index $(k-1)$ is rigid.
\item If $n=2k+2$ and $b=k-1$, then the quadric $Q_X$ has dimension $k+1$. By Lemma \ref{lemma 5.8}, the corank $r= k-1$. We claim that the singular locus of $Q_X$ is the desired linear space of projective dimension $k-2$ that is contained in every $k$-plane parametrized by $X$.

Let $\mathbb{P}(F_{k-1})$ be the singular locus of $Q_X$. Notice that the maximal linear subspaces contained in $Q_X$ have projective dimension $k$. Let $F$ be a maximal isotropic subspace of dimension $k+1$ that is contained in $Q_X$ and contains at least one $k$-plane parametrized $X$. Let $X_F$ be the locus of $k$-planes parametrized by $X$ that are contained in $F$. Then by Lemma \ref{lemma 5.7} $$[X_F]=\sigma_{1,2,...k-1,k+1}\in A(G(k,k+1)).$$ By Theorem \ref{index}, the subindex $(k-1)$ is rigid and therefore there exists a unique linear space $V_{k-1}^F$ of dimension $k-1$ which is contained in every $k$-plane parametrized by $X_F$. We claim that $V_{k-1}^F$  is the singular locus $F_{k-1}$. 

Suppose for a contradiction that $V_{k-1}^F\neq F_{k-1}$. Let $p$ be a general point in $F$ which is not contained in $V_{k-1}^F$ and $F_{k-1}$. Let $W$ be the span of $p$ and $F_{k-1}$. Let $F'$ be the other $(k+1)$-plane contained in the orthogonal complement of $W$ which belongs to different connected component than $F$. Notice that then there are at least two $k$-planes parametrized by $X$ and containing $p$, namely the span of $p$ and $V_{k-1}^F$ or $V_{k-1}^{F'}$. It contradicts the fact that through a general point there is only one $k$-plane parametrized by $X$. We conclude that $V_{k-1}^F=F_{k-1}$. Since every $k$-plane parametrized by $X$ is contained in some maximal isotropic subspace, it must contain $F_{k-1}$. Therefore $(k-1)$ is rigid.

\item If $n>2k+2$ is odd and $b=[\frac{n}{2}]-1=\frac{n-3}{2}$, we will use induction on $n$ and Lemma \ref{5.5}. By Lemma \ref{lemma 5.8}, the quadric $Q_X$ has corank $k-1\leq r\leq b=\frac{n-3}{2}$. There are two possibilities:
\begin{enumerate}
\item If $k-1\leq r<b$, then $X$ can be viewed as a subvariety of $OG(k,n-b+r)$ where $n-b+r<n$. Let $i:OG(k,n-b+r)\hookrightarrow OG(k,n)$ be the natural inclusion. Since $[i(X)]=\sigma_{1,...,k-1}^{b}$, $[X]=\sigma_{a'}^{b'}\in A(OG(k,n-b+r))$ is a Schubert class. Since $Q_X$ has dimension $$n-b-2=n-b+r-b'-2,$$ we obtain $b'=r$. Compare the dimension of $\sigma_{a';b'}$ and $\sigma_{a;b}$, we must have $$a'=(1,...,k-1).$$ By inducsion on $n$, $a_{k-1}=k-1$ is rigid. 

\item Now suppose the corank $r=b$. Let $Z=\mathbb{P}(F_b)$ be the singular locus of $Q_X$. Notice that a maximal isotropic subspace contained in $Q_X$ has a vector space dimension $b+1$ and $F_b$ is contained in all maximal isotropic subspaces. Then every $k$-plane $\Lambda$ parametrized by $X$ must meet $F_b$ in dimension at least $k-1$ since the span of $\Lambda$ and $F_b$ is contained in a maximal isotropic subspace. Let $Y$ be the subvariety defined in Lemma \ref{5.5}. Then $$\dim(Y)\leq\dim(Z)= b-1$$ and 
by Lemma \ref{5.5}, $\dim(Y)=k-2$. Since $Y$ is irreducible, $Y$ has to be a linear space. Let $Y=\mathbb{P}(F_{k-1})$, then $F_{k-1}$ is the desired vector subspace that is contained in every $k$-planes parametrized by $X$. 
\end{enumerate}

\item If $n>2k+2$ is even and $b=\frac{n}{2}-1$, under the natural inclusion $i:OG(k,n)\hookrightarrow G(k,n)$, 
$$[i(X)]=\sigma_{1,...,k-1,b+1}\in A(G(k,n)).$$ By Theorem \ref{index}, $a_{k-1}=k-1$ is rigid. 

\item If $n>2k+2$ is even and $b=\frac{n}{2}-2$, we will use a similar approach as in Example \ref{example1}. The quadric $Q_X$ has corank $r\leq \frac{n}{2}-2$. 

If $r<b=\frac{n}{2}-2$, then a similar argument as above shows that $a_{k-1}=k-1$ has to be rigid. 

Now assume the $r=b$. Let $Z=\mathbb{P}(F_b)$ be the singular locus of $Q_X$. We claim that every $k$-plane parametrized by $X$ must meet $Z$ in projective dimension at least $k-2$ and then we conclude $a_{k-1}=k-1$ is rigid by Lemma \ref{5.5}. 

To prove the claim, let $\Lambda$ be a general point in $X$ and let $F$ be a maximal isotropic subspace in $Q_X$ containing $\Lambda$. Let $X_F$ be the locus of $k$-planes parametrized by $X$ that are contained in $F$. Then by Lemma \ref{lemma 5.7} $$[X_F]=\sigma_{1,...,k-1,b+2}\in A(G(k,b+2)).$$ The sub-index $(k-1)$ is rigid by Theorem \ref{index}. Therefore there is a unique $V^F_{k-1}$ that is contained in all $k$-planes parametrized by $X_F$. Meanwhile it also implies $V^F_{k-1}$ is contained in the singular locus $Z$ since otherwise through a general point in $F$ there are at least two $k$-planes parametrized by $X$, which is impossible. As we vary $F$, and since every $k$-plane parametrized by $X$ is contained in some maximal isotropic subspace, we get every $\Lambda\in X$ must meet $Z$ in projective dimension at least $k-2$. This complete the proof of the claim.

\item Finally, if $n\geq 2k+2$ and $b<[\frac{n-1}{2}]-1$, we will use induction on $b$. Consider the incidence correspondence
$$I:=\{(\Lambda,H)|\Lambda\in X,\Lambda\subset H,H\in(\mathbb{P}^{n-1})^*\}.$$
Let $\pi_1:I\rightarrow X$ and $\pi_2\rightarrow(\mathbb{P}^{n-1})^*$ be the two projections. Since $X$ is irreducible and the fibers of $\pi_1$ are projective linear spaces of dimension $n-k-1$, $I$ is irreducible of dimension $2n-2k-b-2$, and its image $\pi_2(I)$ in $(\mathbb{P}^{n-1})^*$ is also irreducible. Since a general fiber of $\pi_2$ over the image has cohomology class
$$[\pi^{-1}_2(H)]=\sigma_{1,...,k-1}^{b+1},$$
we get $$\dim(\pi_2(I))=2n-2k-b-2-(n-k-b-2)=n-k.$$ Hence $$[\pi_2(I)]=m\sigma_{1,...,k-1,k+1,...,n}\in A(G(n-1,n).$$ We claim that $m=1$. Let $F_{n-k}$ be a general linear spacce of dimension $n-k$ and let $$\Sigma=\Sigma_{n-k,n-k+2,...,n}=\{\Lambda\in G(k,n)|\dim(\Lambda\cap F_{n-k})\geq 1\}$$ be the Schubert variety in $G(k,n)$ defined by $F_{n-k}$. Then 
$$[X\cap\Sigma]=\sigma_{1,...,k-1;b+1}\in A(OG(k,n)).$$
By induction on $b$, there is a unique linear space $F'_{k-1}$ that is contained in every $\Lambda\in X\cap \Sigma$, and hence the span of $F'_{k-1}$ and $F_{n-k}$ is the unique hyperplane that contains $\mathbb{P}(F_{n-k})$ and is contained in $\pi_2(I)$. Therefore $m=1$. By Theorem \ref{index}, there is a unique linear subspace $F_{k-1}$ such that $$\pi_2(I)=\{H\in(\mathbb{P}^{n-1})^*|F_{k-1}\subset H\}.$$
Now clearly $F_{k-1}$ is contained in every $k$-plane parametrized by $X$, since otherwise there exists a hyperplane in $\pi_2(I)$ that does not contain $F_{k-1}$. 
\end{itemize}

\end{proof}

\begin{remark}
When $n=2k+1$, let $X$ be the restriction variety defined by
$$F_1\subset...\subset F_{k-2}\subset F_k\subset Q_{k+2}^{k-2}.$$
Then $[X]=\sigma_{1,...,k-1;k-1}$. This gives a counterexample that $a_{k-1}=k-1$ is not rigid when $n=2k+1$.
\end{remark}

\begin{corollary}\label{1 is rigid 2}
Let $\sigma_{a;b}$ be a Schubert class in $OG(k,n)$, where $a=(1,...,t,t+2,...,k)$ and $b=t$, $1\leq t\leq k-2$. Then the sub-index $a_t=t$ is rigid when $n\geq 2k+2$.
\end{corollary}
\begin{proof}
Let $X$ be a subvariety of $OG(k,n)$ representing the Schubert class $\sigma_{a;b}$. We will use a similar argument as in the proof of the last case of Proposition \ref{1 is rigid}.

Consider the incidence correspondence
$$I:=\{(\Lambda,H)|\Lambda\in X,\Lambda\subset H,H\in(\mathbb{P}^{n-1})^*\}.$$
Let $\pi_1:I\rightarrow X$ and $\pi_2\rightarrow(\mathbb{P}^{n-1})^*$ be the two projections. Since $X$ is irreducible of dimension $n-t-k-1$ and the fibers of $\pi_1$ are projective linear spaces of dimension $n-k-1$, $I$ is irreducible of dimension $2n-2k-t-2$ and its image $\pi_2(I)$ in $(\mathbb{P}^{n-1})^*$ is also irreducible. Let $i:OG(k,n)\rightarrow G(k,n)$ be the natural inclusion. Then $$[i(X)]=2\sum\limits_{i=1}^{k-t}\sigma_{1,...,k-i,k-i+2,...,k,n-t-i},$$ and hence for a general $H\in \pi_2(I)$,
$$[i(\pi^{-1}_2(H))]=2\sigma_{1,...,k-1,n-k-1},$$
which has dimension $n-2k-1$. Therefore $\dim(\pi_2(I))=2n-2k-t-2-(n-2k-1)=n-t-1$. Hence $$[\pi_2(I)]=m\sigma_{1,...,t,t+2,...,n}\in A(G(n-1,n).$$ We claim that $m=1$. Let $F_{n-t-1}$ be a general linear spacce of dimension $n-t-1$ and let $$\Sigma=\Sigma_{n-k,...,n-t-1,n-t+1,...,n}=\{\Lambda\in G(k,n)|\dim(\Lambda\cap F_{n-t-1})\geq k-t\}$$ be the Schubert variety in $G(k,n)$ defined by $F_{n-t-1}$. Then 
$$[X\cap\Sigma]=\sigma_{1,...,k-1;k}\in A(OG(k,n)).$$
By Proposition \ref{1 is rigid}, there is a unique $F_{k-1}$ that is contained in every $\Lambda\in X\cap \Sigma$, and hence the span of $F_{k-1}$ and $F_{n-t-1}$ (they have a $(k-t-1)$-dimensional intersection by the construction) is the unique hyperplane that contains $\mathbb{P}(F_{n-t-1})$ and is contained in $\pi_2(I)$. Therefore $m=1$. By Theorem \ref{index}, there is a unique linear subspace $F_{t}$ such that $$\pi_2(I)=\{H\in(\mathbb{P}^{n-1})^*|F_{t}\subset H\}.$$
Now clearly $F_{t}$ is contained in every $k$-plane parametrized by $X$, since otherwise there exists a hyperplane in $\pi_2(I)$ that does not contain $F_{t}$. 
\end{proof}

\begin{remark}
We will show it later that Corollary \ref{1 is rigid 2} is also true when $n=2k$. However, it would be false when $n=2k+1$. For example, the restriction variety defined by $$F_2\subset F_3\subset Q_6^0$$ has cohomology class $\sigma_{1,3;1}\in OG(3,7)$. 
\end{remark}

\begin{corollary}\label{rigid k-1}
Assume $k\geq 2$. Let $\sigma_{a;b}$ be a Schubert class in $OG(k,n)$ where $a=(1,...,t,t+2,...,k)$ and $b=t$, $0\leq t\leq k-2$. Then the sub-index $a_{k-1}=k$ is rigid.
\end{corollary}
\begin{proof}
Let $X$ be a subvariety representing $\sigma_{a;b}$ in $OG(k,n)$. Let $Q_X$ be the quadric swept out by the $k$-planes parametrized by $X$. Let $r$ be the corank of $Q_X$. We will consider the cases when $n=2k,2k+1,2k+2$ and $n\geq 2k+3$ separately. 
\begin{itemize}
\item If $n=2k$, then Algorithm \ref{algorithm2} implies $X$ can not be contained in a smooth quadric smaller than $Q$. Therefore $r=t$. Let $\mathbb{P}(F_t)$ be the singular locus of $Q_X$. Notice that $F_t$ is contained in every maximal isotropic subspace that is contained in $Q_X$, and under the projection from $F_t$, we reduce to the case when $t=0$. Robles and The prove that $\sigma_{2,...,k;0}$ is multi-rigid \cite{RT}. In particular, $\sigma_{a;b}$ is rigid.

\item If $n=2k+1$, we first reduce to the case when $t=0$, then use induction on $k$. Let $r\leq t$ be the corank of $Q_X$. If $r<t$, then $X$ can be viewed as a subvariety in $OG(k,2k)$ with cohomology class $\sigma_{1,...,t-1,t+1,...,k;t-1}$, which is rigid by the previous case. If $r=t$, let $F_t$ be the singular locus of $Q_X$. Then $F_t$ has to be contained in every maximal isotropic subspace in $Q_X$. Under the projection from $F_t$, we reduce to the case when $t=0$. 

If $k=2$ and $n=5$, then the orthogonal Grassmannian $OG(2,5)$ is isomorphic to the projective space $\mathbb{P}^3$ and every representative of $\sigma_{2}^0\in A(OG(2,5))$ is isotropic to $\mathbb{P}^2$. Therefore the Schubert class $\sigma_{2;0}$ is rigid.

If $k=3$ and $n=7$, observe that the Schubert variety $\Sigma=\Sigma_{2,3;0}(F_3)$ is the cone over the Veronese surface in the minimal embedding of $OG(3,7)$. Let $W_2$ be a $2$-dimensional subspace of $F_3$. Then the line consisting of $3$-planes containing $W_2$ and contained in $W_2^\perp$ is contained in $\Sigma$. Therefore the Schubert variety $\Sigma$ is a cone with the vertex $F_3$. Let $\Sigma_{3;1,0}$ be a general codimension $1$ Schubert variety in $OG(3,7)$ defined with respect to a 3-dimensional linear space $G_3$. Let $Z$ be the intersection of $\Sigma$ and $\Sigma_{3;1,0}$. Then every point $p$ in $\mathbb{P}(G_3)\cong \mathbb{P}^2$ corresponds to a unique line $p^\perp\cap F_3$ contained in $Z$. Therefore $Z$ is isotropic to the Veronese surface and in the minimal embedding of $OG(3,7)$, $\Sigma$ is embedded as a cone over the Veronese surface. We get $\Sigma$, and hence $X$, are varieties of minimal degree. $X$ cannot be a rational normal scroll since the intersection $X\cap \Sigma_{3;1,0}$ contains no lines. Therefore $X$ is a cone over the Veronese surface. Let $V$ be the vertex of $X$. Then the lines in $OG(3,7)$ that contain $V$ sweep out the Schubert variety $\Sigma_{2,3;0}(V)$. This implies $X$ is contained in the Schubert variety $\Sigma_{2,3;0}(V)$ and hence $X$ has to be the Schubert variety.

Now assume the statement is true for $k'<k$. Let $Y$ be a general hyperplane section of $X$. Then $Y$ has cohomology class $\sigma_{1,3,...,k}^1$. Let $Q_Y$ be the corresponding quadric and $r$ its corank. If $r=0$, then by considering $Y$ as a subvariety in $OG(k,2k)$, we get the sub-index $a_{k-1}$ is rigid for $Y$. If $r=1$, let $q$ be the singular point of $Q_Y$. Let $Y'$ be the resulting variety of $Y$ under the projection from $q$. Then $\bar{Y}$ is a subvariety of $OG(k-1,2k-1)$ with cohomology class $\sigma_{2,...,k-1;0}$. By induction, we conclude that the sub-index $a_{k-1}$ is rigid for a general hyperplane section of $X$. 

Let $Y'$ be another hyperplane section of $X$. Let $F_k$ and $F'_k$ be the corresponding linear spaces. We need to show $F_k=F'_k$. Let $Z$ be the intersection of $Y$ and $Y'$. Then $Z$ has cohomology class $\sigma_{1,2,4,...,k;2}$ in $OG(k,n)$. By a similar argument, we conclude $a_{k-1}$ is rigid for $Z$ by reducing to a multi-rigid class when the corank is less than $2$ and by induction when the corank equals $2$. Therefore $F_k$ is independent of the choice of hyperplane sections and thus the sub-index $a_{k-1}$ is rigid for $X$.
\end{itemize}
Now assume $n\geq 2k+2$. By Corollary \ref{1 is rigid 2}, the sub-index $a_t=t$ is rigid. Under the projection from the unique linear space $F_t$, we reduce to the case when $t=0$. 

\begin{itemize}
\item If $n=2k+2$, consider the following incidence correspondence 
$$I:=\{(\Lambda,F_{k+1})|\Lambda\subset F_{k+1},\Lambda\in X, F_{k+1}\text{ isotropic}\}\subset X\times OG(k+1,2k+2).$$
Let $\pi_1:I\rightarrow X$ and $\pi_2:I\rightarrow OG(k+1,2k+2)$ be the two projections. Let $Y$ be the image of $I$ under $\pi_2$. Notice that the fibers of $\pi_1$ have dimension $0$, and therefore $$\dim(Y)\leq \dim(I)=\dim(X)=k+1.$$ Since the $k$-planes parametrized by $X$ sweep out the quadric $Q$, the $(k+1)$-planes parametrized by $Y$ also sweep out $Q$ and therefore $\dim(Y)\geq \dim(\Sigma_{2,...,k+1;0})=k$.

If $\dim(Y)=k+1$, then $[Y]=m\sigma_{2,...,k-1,k+1;k-1,0}$. Let $F_k$ be a general isotropic subspace of dimension $k$. Then there exists a $(k+1)$-plane parametrized by $Y$ that meets the orthogonal complement of $F_k$ in a 3-dimensional subspace, and therefore there exists a $k$-plane parametrized by $X$ that meets the orthogonal complement of $F_k$ in a $2$-dimensional subspace. This is a contradiction since $[X]\cdot\sigma_{k;k,k-2,...,1}=0$. We conclude that $\dim(Y)=k$ and a general fiber of $\pi_2$ has dimension $1$.

Let $F$ be a general point in $Y$. Set $X_F:=\{\Lambda\in X|\Lambda\subset F\}$. Then $[X_F]=m\sigma_{1,...,k-1.k+1}\in A(G(k,F))$. Since through a general point in $Q$, there is only one $k$-plane parametrized by $X$, we get $m=1$. By Theorem \ref{index}, there is a unique linear space $L^F_{k-1}$ that is contained in every $k$-plane parametrized by $X_F$.

Let $Z$ be an isotropic subspace or the orthogonal complement of an isotropic subspace such that $\dim(Z\cap \Lambda)\geq k-1$ for all $\Lambda\in X$. Assume $Z$ is of the minimal dimension among all such subspaces. Then for a general $(k-1)$-dimensional isotropic subspace $V_{k-1}$ contained in $Z$, there exists $\Lambda\in X$ that containes $V_{k-1}$. Let $F$ be one component of the intersection of $Q$ and the orthogonal complement of $V_{k-1}$. Notice that $L^F_{k-1}=V_{k-1}$, and hence there is at least one dimensional family of $k$-planes parametrized by $X$ that contains $V_{k-1}$. Therefore there is at most $k$-dimensional family of $(k-1)$-planes contained in $Z$. Hence $Z$ is an isotropic subspace of dimension $k$ such that $\dim(Z\cap \Lambda)\geq k-1$ for all $\Lambda\in X$.

For the uniqueness, suppose $Z'$ is another isotropic subspace of dimension $k$ such that $\dim(Z\cap \Lambda)\geq k-1$ for all $\Lambda\in X$. Then every $k$-plane parametrized by $X$ must be contained in the span of $Z$ and $Z'$ which has projective dimension at most $2k-1$. This contradicts the fact that the $k$-planes parametrized by $X$ sweep out the quadric $Q$ which is of dimension $2k$.

\item Now assume $n>2k+2$. For a general hyperplane $H$, let $$X_H:=\{\Lambda\in X|\Lambda\subset H\}.$$ Then $$[X_H]=\sigma_{1,...,k-1;k}\in A(OG(k,n)).$$ By Proposition \ref{1 is rigid}, the sub-index $(k-1)$ is rigid. Let $F_{k-1}^H$ be the unique $(k-1)$-dimensional linear space that is contained in every $k$-plane parametrized by $X_H$. Let $Y$ be the variety swept out by $\mathbb{P}(F_{k-1}^H)$ as we vary $H$. We claim that the variety $Y$ is a projective linear space of dimension $k-1$.

Suppose for a contradiction that $\dim(Y)\geq k$, then for a general linear space $G_k$, the variety $Y$ will meet $\mathbb{P}(G_k^\perp)$, i.e. we can find a hyperplane $H$ such that $$\dim(F_{k-1}^H\cap G_k^\perp)\geq 1.$$ Let $Q_H$ be the quadric swept out by the $k$-planes parametrized by $X_H$, then by \cite{Coskun2014RigidityOS} Lemma 6.2, $\dim(Q_H)=n-k-2$ and therefore 
$$\dim(Q_H\cap \mathbb{P}(G^\perp_k))\geq 1.$$ 
Let $p_1$ be a point contained in $\mathbb{P}(F^H_{k-1})\cap G_k^\perp$ and $p_2$ be a point different from $p_1$ and contained in $Q_H\cap \mathbb{P}(G^\perp_k)$. Then there exists a $k$-plane $\Lambda\in X_H$ that contains both $p_1$ and $p_2$. We obtain $$\dim(\Lambda\cap G^\perp_k)\geq 2,$$which is a contradiction since $\sigma_{a;b}\cdot \sigma_{;k+1,k,k-2,...,1}=0$. Therefore $\dim(Y)=k-1$.

If $\dim(Y)=k-1\geq 2$, then by Bertini's theorem, a general hyperplane section of $Y$ is irreducible. Since $\mathbb{P}(F_{k-1}^H)\subset Y\cap H$, we must have $\mathbb{P}(F_{k-1}^H)=Y\cap H$, and therefore $Z$ is linear. 

If $\dim(Y)=1$ and $k=2$, suppose for a contradiction that $\deg(Y)=d\geq 2$. Let $H$ be a general hyperplane, then $H$ will meet $Y$ in $d$ projective points $p_1,...,p_d$. Assume that $p_1=\mathbb{P}(F_1^H)$ is the point defined by $H$. Let $H'\neq H$ be another hyperplane that defines $p_2=\mathbb{P}(F_1^{H'})$ and let $X_{H\cap H'}$ be the locus of $k$-planes parametrized by $X$ that are contained in $H\cap H'$. Then
$$[X_H]=\sigma_1^2$$
and therefore $$\dim(X_{H\cap H'})\geq 1.$$
Since every line contained in $X_{H\cap H'}$ must contain both $p_1$ and $p_2$, we reach a contradiction. We conclude that $\deg(Y)=1$. Let $Y=\mathbb{P}(F_k)$. Then $F_k$ is the desired linear space that meets every $k$-plane parametrized by $X$ in dimension at least $k-1$. 

\end{itemize}
\end{proof}

Now we can prove that the essential sub-index $a_t=t$ is always rigid.
\begin{proposition}
Assume $n\geq 2k+2$. Let $\sigma_{a;b}$ be a Schubert class in $OG(k,n)$. Suppose that the sub-index $a_t=t$ is essential (i.e. $a_{t+1}\neq t+1$), then $a_t$ is rigid.
\end{proposition}
\begin{proof}
We will prove it by induction on $k-s$. If $k-s=0$, then $a_t$ is rigid by Theorem \ref{index}. 

Now assume $k-s\geq1$. Let $X$ be a subvariety representing the Schubert class $\sigma_{a;b}$ in $OG(k,n)$. By \cite{Coskun2014RigidityOS} Lemma 6.2, the $k$-planes parametrized by $X$ sweep out a quadric $Q_X$ of dimension $n-b_1-2$. For a general point $q\in Q_X$, let $X_q$ be the locus of $k$-planes parametrized by $X$ that contains $q$. Then $[X_q]=\sigma_{a';b'}$, where $a'_1=1$, $a'_{i+1}=a_{i}+1$ if $a_i\leq b$ and $a'_{i+1}=a_i$ if $a_i>b$ for $1\leq i\leq s$, $b'_{j+1}=b_{j}$ for $s+1\leq j\leq k-1$.

If $b_1\neq t$ or $a_{t+1}-a_t\geq 3$, then $a'_{t+1}=t+1$, $a'_{t+2}\neq t+2$ and $|b'|<|b|$. By induction, there exists a unique $(t+1)$-dimensional space $F^q_{{t+1}}$ that is contained in every $k$-plane parametrized by $X_q$. Notice that $q$ has to be contained in $F^q_{t+1}$ since otherwise $F^q_{t+1}$ can be replaced by any codimension $1$ linear subspace of the span of $F^q_{t+1}$ and $q$, which contradicts the uniqueness. Let $I$ be the Zariski closure of the locus of all possible pairs $\{(q,F^q_{t+1})\}$ in $Q_X\times OG(t+1,n)$.

Let $\pi_1:I\rightarrow Q_X$ and $\pi_2:I\rightarrow OG(t+1,n)$ be the two projections. Specializing $X$ to a Schubert variety specializes $\pi_2(I)$ to a Schubert variety $\Sigma_{1,...,t;b_1}$. Therefore 
$$[\pi_2(I)]=\sigma_{1,...,t;b_1}\in A(OG(t+1,n)).$$ 
By Proposition \ref{1 is rigid}, there exists a unique isotropic subspace $F_{t}$ that is contained in all $F^q_{t+1}$, which in turn shows that $F_t$ is contained in every $k$-plane parametrized by $X$. 

The proof of the case when $b_k=t$ and $a_{t+1}-a_t= 2$ is almost the same except that for a general point $q\in Q_X$, we identify the unique linear space $F^q_{t'+1}$ corresponding to $X_q$ where $t'$ is the sub-index such that $a'_{t'}=t'$ and $a'_{t'+1}\neq t'+1$. The image $\pi_2(I)\subset OG(t'+1,n)$ has cohomology class $\sigma_{1,...,t,t+2,...,t';t}$ which will lead to the same result by Corollary \ref{1 is rigid 2}.

\end{proof}

Then we characterize the rigid sub-indices in $b_\bullet$.

\begin{proposition}
\label{quadric rigd}
The sub-index $b_1$ is rigid if and only if either $b_1=0$ or there exist $i$ and $j$ such that $a_i=b_j$ and $x_j>k-j+b_j-\frac{n-3}{2}$, where $x_j:=\#\{i|a_i\leq b_j\}$. 
\end{proposition}
\begin{proof}
The proof is due to the observation that $b_1$ is not rigid if and only if we can find a deformation in $OG(k,n-1)$. Let $X$ be a subvariety in $OG(k,n)$ representing the Schubert class $\sigma_{a;b}$. Let $Q_X$ be the quadric swept out by the $k$-planes parametrized by $X$. Let $r$ be the corank of $Q_X$.

It is clear that if $b_1=0$, then it is rigid since the quadric $Q_X$ has to be the whole quadric $Q$. Now assume $b_1>0$. If $r=b_1$, let $\mathbb{P}(F_{b_1})$ be the singular locus of $Q_X$. Then $Q_X=Q\cap \mathbb{P}(F^\perp_{b_1})$ and every $k$-plane parametrized by $X$ is contained in $F^\perp_{b_1}$. Therefore $b_1$ is rigid.

If $r\neq b_1$, then $Q_X$ is contained in a smooth quadric of dimension $n-2$ and therefore $X$ can be viewed as a subvariety of $OG(k,n-1)$. Notice that the cohomology class of $X$ considered as a subvariety of $OG(k,n-1)$ is also a Schubert class $\sigma_{a'}^{b'}$. A Schubert variety $\Sigma_{a'}^{b'}$ in $OG(k,n-1)$ can be viewed as a restriction variety $R$ in $OG(k,n)$ defined by a partial flag:
$$G_{a'_1}\subset...\subset G_{a'_s}\subset {Q'}_{n-1-b'_{k-s}}^{b'_{k-s}}\subset...\subset {Q'}_{n-1-b'_1}^{b'_1}.$$
Therefore $b_1$ is not rigid if and only if the cohomology class of $R$ is the Schubert class $\sigma_{a;b}$.

Applying the Algorithm \ref{algorithm2}, if the cohomology class $[R]$ is a Schubert class, then at each stage, it should return either $D^a$ or $D^b$. 

It always returns $D^a$ if and only if $a'_i\neq b'_j+2$ for all $1\leq i\leq s$ and $1\leq j\leq k-s$. Notice that $D^a$ is obtained from $D$ by increasing the corank of every quadric by $1$ and is always admissable, we get $a_i=a'_i$, $b_j=b'_j+1$ and therefore $a_i\neq b_j$ for all $i,j$.

If it returns $D^b$ at some stages, assume it first does at the $j$-th quadric, then $n$ has to be odd and 
\begin{eqnarray}
x_j=k-j+1-\frac{n-2b_j'-1}{2}.
\end{eqnarray}
At all the previous stages, it will return $D^a$, which gives $a'_i\neq b'_{j'}+2$ for all $1\leq i\leq s$ and $1\leq j'\leq j$. Equation (1) guarantees that at all the succeeding stages, it will return $D^a$ for all non-essential quadric and will return $D^b$ for every essential quadric. In our case $D^b$ is always admissable and therefore $b_j=b'_j+1$, $a_i=b'_j+1$ if $b'_j$ is essential, $a_i-b'_{j+t}=2$ and $b'_{j+t}-b'_j=t$, and $a_i=a'_i$ for all other $i$. Consequently, we get $b_1$ is not rigid if and only if either $a_i\neq b_j$ for all $i,j$ or $n$ is odd and for all $j$ such that $b_j=a_i$ for some $i$, $x_j=k-j+1-\frac{n-2b_j-1}{2}$.

\end{proof}

As an application, we show that Corollary \ref{1 is rigid 2} is also true when $n=2k$. 

\begin{corollary}
Assume $n= 2k$. Let $\sigma_{a;b}$ be a Schubert class in $OG(k,2k)$, where $a=(1,...,t,t+2,...,k)$ and $b=t$, $1\leq t\leq k-2$. Then the sub-index $a_t=t$ is rigid.
\end{corollary}
\begin{proof}
By Proposition \ref{quadric rigd}, the sub-index $b=t$ is rigid. Let $X$ be a subvariety representing the Schubert class $\sigma_{a;b}$. Then the $k$-planes parametrized by $X$ sweep out a quadric $Q_X$ of dimension $n-t-2$ with corank $t$. Let $F_t$ be the singular locus of $Q_X$. It is clear that every maximal isotropic subspace contained in $Q_X$ must contain the singular locus $F_t$. Therefore the sub-index $a_t$ is rigid.
\end{proof}

\begin{proposition}\label{quadric rigid}
Let $\sigma_{a;b}$ be a Schubert class in $OG(k,n)$. Assume the sub-index $b_j$ is essential. Then the sub-index $b_j$ is rigid if and only if there exist $1\leq i\leq s$ and $j\leq j'\leq k-s$ such that $a_i=b_{j'}$ and $x_{j'}>k-j'+b_{j'}-\frac{n-1}{2}$.
\end{proposition}
\begin{proof}
First assume all conditions in the proposition hold and we will show the sub-index $b_j$ is rigid. The proof will be done by induction on $j$. 

If $j=1$, then the proposition reduces to Proposition \ref{quadric rigd}. Suppose the proposition is true for $j''<j$. Let $X$ be a subvariety representing $\sigma_{a;b}$ in $OG(k,n)$. Then the $k$-planes paremetrized by $X$ sweep out a quadric $Q_X$ of dimension $n-b_1-2$. Let $p$ be a general point of $Q_X$ and define $X_p:=\{\Lambda\in X|p\in \Lambda\}$. Then $X_p$ has cohomology class $\sigma_{a'}^{b'}$ where $a'_1=1$, $a'_{i+1}=a_i+1$ if $a_i\leq b_1$ and $a'_{i+1}=a_i$ if $a_i>b+1$ for $1\leq i\leq s-1$ and $b'_{\gamma+1}=b_\gamma$ for $1\leq \gamma\leq k-s-1$. 

Notice that the conditions in the proposition hold for $(a;b)$ if and only if they also hold for $(a';b')$. By induction, the sub-index $b'_{j-1}$ is rigid with respect to $(a';b')$. Let $F^p_{b_j}$ be the corresponding linear space and $I$ be the Zariski closure in $Q_X\times OG(b_j,n)$ of the locus of all the possible pairs $(p,F^p_{b_j})$. 

Let $\pi:I\rightarrow OG(b_j,n)$ be the natural projection and let $Y=\pi(I)$ be its image. Specializing $X$ to a Schubert variety specializes $Y$ to a Schubert variety $\Sigma_{1,...,b_1,b_1+2,...,b_j;b_1}$. Therefore $$[Y]=\sigma_{1,...,b_1,b_1+2,...,b_j;b_1}.$$ 
By Corollary \ref{rigid k-1}, there exists a unique isotropic subspace $F_{b_j}$ such that $\dim(F_{b_j}\cap F^p_{b_j})=b_j-1$ for a general $F^p_{b_j}\in Y$. Now for a general $\Lambda\in X$, we can find $p\in \Lambda$ such that $\dim(\Lambda\cap (F^p_{b_j})^\perp)=k-j+2$ and therefore $$\dim(\Lambda\cap (F_{b_j})^\perp)\geq k-j+1.$$ By the semi-continuity of the dimension of the intersection, we conclude that $b_j$ is rigid.

On the other hand, suppose first that $a_i\neq b_j'$ for all $1\leq i\leq s$ and $j\leq j'\leq k-s$. Consider the restriction variety $\Gamma$ defined by
$$F_{a_1}\subset...\subset F_{a_s}\subset Q_{n-b_{k-s}}^{b_{k-s}-1}\subset...\subset Q_{n-b_j}^{b_{n_j}-1}\subset Q_{n-b_{j-1}}^{{b_{j-1}}}\subset...\subset Q_{n-b_1}^{b_1}.$$
Applying Algorithm \ref{algorithm2}, we get $[\Gamma]=\sigma_{a;b}$ and therefore $b_j$ is not rigid.

Now suppose $n$ is odd and let $J$ be the set consisting of all $j\leq j'\leq k-s$ such that $b_{j'}=a_i$ for some $1\leq i\leq s$. Assume that $J\neq\emptyset$ and $x_{j'}=k-j'+b_{j'}-\frac{n-1}{2}$ for all $j'\in J$. Let $j_0=\min\{J\}$. Consider the restriction variety $\Phi$ defined by 
$$F_{a'_1}\subset...\subset F_{a'_s}\subset Q_{n-b_{k-s}}^{b_{k-s}-1}\subset...\subset Q_{n-b_j}^{b_{n_j}-1}\subset Q_{n-b_{j-1}}^{{b_{j-1}}}\subset...\subset Q_{n-b_1}^{b_1},$$
where $a'$ is obtained from $a$ by changing all the $a_i$ such that $a_i\geq b_{j_0}$ to the maximal admissible sequence with respect to $b_{k-s}-1,....,b_{j_0}-1$ (i.e. the sequence consists of all the numbers from $b_{j_0}+1$ to $\left[\frac{n}{2}\right]$ excluding the numbers equal to $b_{j'}$). Apply Algorithm \ref{algorithm2}, $[\Phi]=\sigma_{a;b}$ and therefore $b_j$ is not rigid.
\end{proof}

Next we characterize rigid subindices in $a_\bullet$.

\begin{proposition}
Let $\sigma_{a;b}$ be a Schubert class in $OG(k,n)$. Assume $a_i$ is an essential sub-index such that $a_i\neq b_j$ for all $1\leq j\leq k-s$. Then the sub-index $a_i$ is rigid if either $a_i=a_{i-1}+1$ or $a_{i+1}-a_{i}\geq 3$ and $z_i\neq a_{i+1}-a_i-2$, where $z_i:=\#\{j|a_i\leq b_j<a_{i+1}\}$. 
\end{proposition}

\begin{proof}
We will use induction on $k-s$. If $k-s=0$, then the proposition reduces to Theorem \ref{index}.

Now assume the proposition holds for $k-s<\gamma$. Let $X$ be a subvariety representing $\sigma_{a;b}$ and $Q_X$ be the quadric swept out by the $k$-planes parametrized by $X$. For a general point $p\in Q_X$, define $X_p:=\{\Lambda\in X|p\in \Lambda\}$. Then $X_p$ has cohomology class $\sigma_{a'}^{b'}$ where $a'_1=1$, $a'_{\mu+1}=a_\mu+1$ if $a_\mu\leq b_1$ and $a'_{\mu+1}=a_\mu$ if $a_i>b+1$ for $1\leq \mu\leq s-1$ and $b'_{j+1}=b_j$ for $1\leq j\leq k-s-1$. Notice that if the conditions in the proposition hold for $(a;b)$, they also hold for $(a';b')$. By induction, $a'_{i+1}$ is rigid.

If $b_1\geq a_i$, then $a'_{i+1}=a_i+1$. Let $F^p_{a_i+1}$ be the corresponding linear space and let $I$ be the Zariski closure in $Q_X\times OG(a_i+1,n)$ of the locus of all the possible pairs $(p,F^p_{a_i+1})$. 

Let $Y$ be the image of $I$ under the projection $I\rightarrow OG(a_i+1,n)$. By specializing $X$ to a Schubert variety, we get $[Y]=\sigma_{1,...,a_i;b_1}$. By Proposition \ref{1 is rigid}, there is a unique isotropic subspace $F_{a_i}$  such that $F_{a_i}\subset F_{a_i+1}^p$ for all $F^p_{a_i+1}\in Y$. Therefore for a general $\Lambda\in X$, $\dim(\Lambda\cap F^p_{a_{i}+1})=i+1$ for some $p\in \Lambda$, and hence $\dim(\Lambda\cap F_{a_i})\geq i$. By semicontinuity, $\dim(\Lambda\cap F_{a_i})\geq i$ for every $\Lambda\in X$. The uniqueness of $F_{a_i}$ is guaranteed by Proposition \ref{1 is rigid}, and therefore the sub-index $a_i$ is rigid. The proof of the case when $b_1<a_i$ is almost identical except the cohomology class $[Y]=\sigma_{1,...,b_1,b_1+2,...,a_i;b_1}\in A(OG(a_i,n))$. By Proposition \ref{rigid k-1}, we conclude that $a_i$ is rigid.
\end{proof}

\begin{remark}\label{notrigid}
If $a_i=a_{i-1}+1=a_{i+1}-2$ and $a_i\neq b_j$ for all $1\leq j\leq k-s$, then $a_i$ is not rigid. For a proof, see \cite{Coskun2014RigidityOS} Theorem 1.7(2).
\end{remark}

\begin{proposition}\label{isonotrigid}
Assume that $a_i\neq b_j$ for all $1\leq j\leq k-s $, $a_i\neq a_{i-1}+1$ and $z_i=a_{i+1}-a_i-2\geq 1$, where $z_i:=\#\{j|a_i< b_j<a_{i+1}\}$. Then $a_i$ is not rigid.
\end{proposition}
\begin{proof}
We will construct a subvariety representing $\sigma_{a;b}$ such that there is no $F_{a_i}$ that meets all $k$-planes parametrized by $X$ in dimension at least $i$.

The equality $z_i=a_{i+1}-a_i-2\geq 1$ implies that the sequence $b$ contains $a_i+1,...,a_{i+1}-2$. Say $b_\gamma=a_i+1$. First assume that $\gamma=1$. Choose a complete isotropic flag $F_\bullet$. In $OG(k-z_i,n)$, consider the Schubert index $(a';b')$ defined by $a'_\mu=a_\mu+z_i$ for $1\leq \mu\leq i$, $a'_{\mu}=a_{\mu}$ for $i+1\leq \mu\leq s$, $b'_j=b_{j+z_i}$ for $1\leq j\leq k-s-z_i$. Then $a'_i=a_i+z_i$ is not rigid by Remark \ref{notrigid}. Let $Y$ be a subvariety of $OG(k-z_i,n)$, but not a Schubert variety, such that $[Y]=\sigma_{a';b'}$, $\dim(\Lambda\cap F_{a'_\mu})\geq \mu$ for $\mu\neq i$, $\dim(\Lambda\cap F_{b'_j}^\perp)\geq k-j+1$ for $1\leq j\leq k-s-z_i$ and for all $\Lambda\in Y$. Let $X$ be the Zariski closure in $OG(k,n)$ of the following locus of $k$-planes:
$$\{\text{span}\{G_{z_i},\Lambda\}|G_{z_i}\text{ is a linear subspace contained in } F_{a_i+1}^\perp\backslash F_{a_{i+1}}^\perp, \Lambda\in Y,\Lambda\subset G_{z_i}^\perp\}. $$
By specializing $Y$ to a Schubert variety, we can see $[X]=\sigma_{a;b}$. Furthermore, $X$ is not isomorphic to a Schubert variety.

Now assume $\gamma\geq 2$. Let $Z$ be a non-Schubert subvariety of $OG(k-\gamma+1,n)$ representing $\sigma_{a';b'}$, where $a'=a$, $b'_j=b_{j+\gamma-1}$ for $1\leq j\leq k-s-\gamma+1$, such that there does not exist a linear space of dimension $a_i$ that meet all $k-\gamma+1$-planes parametrized by $Z$ in dimension at least $i$. Let $T$ be the Zariski closure of 
$$\{\Lambda\in OG(k,n)|\dim(\Lambda\cap F_{b_j}^\perp)\geq k-j+1, 1\leq j<\gamma,\Lambda\cap F_{b_\gamma}\in Y\}.$$
By specializing $Z$ to a Schubert variety, we can see $[T]=\sigma_{a;b}$. Since $T$ is not a Schubert variety, we conclude that $a_i$ is not rigid.
\end{proof}

Now we consider the case when $a_i=b_j$ for some $1\leq j\leq k-s$.

\begin{proposition}\label{isorigid1}
Assume that $a_i=b_j$ for some $1\leq j\leq k-s$. Then $a_i$ is rigid if and only if $x_{j}>k-j+b_{j}-\frac{n-3}{2}$.
\end{proposition}

\begin{proof}
If $x_{j}=k-j+b_{j}-\frac{n-3}{2}$, then the proof of Proposition \ref{quadric rigid} gives a counter-example which shows $a_i$ is not rigid.

Now assume $x_{j}>k-j+b_{j}-\frac{n-3}{2}$. Let $X$ be a subvariety representing the Schubert class $\sigma_{a;b}$. By Proposition \ref{quadric rigid}, $b_j$ is rigid. Let $Q_{n-a_i}^{a_i}$ be the corresponding quadric with singular locus $F_{a_i}$. Consider $Y:=\{\Lambda\in X|\dim(\Lambda\cap F_{a_i})\geq i\}$. By specializing $X$ to a Schubert variety, we get $Y=X$, i.e. $\dim(\Lambda\cap F_{a_i})\geq i $ for every $\Lambda\in X$. For the uniqueness, suppose, for a contradiction, that there is another $G_{a_i}\neq F_{a_i}$ such that $\dim(\Lambda\cap G_{a_i})\geq i $, $\forall\Lambda\in X$. Let $p$ be a general point in $\mathbb{P}(F_{a_i})$ not contained in $\mathbb{P}(G_{a_i})$. Define $X_p:=\{\Lambda\in X|p\in \Lambda\}$. Then $[X_p]=\sigma_{a';b'}$ where $a_1=1$, $a'_\mu=a_{\mu-1}+1$ for $2\leq \mu \leq i$, $a'_\mu=a_\mu$ for $i<\mu\leq s$ and $b'_j=b_j$ for $1\leq j\leq k-s$. On the other hand, let $G_{a_i+1}$ be the span of $p$ and $G_{a_i}$. By assumption, $\dim(\Lambda\cap G_{a_i+1})\geq i+1$ for every $\Lambda\in X_p$. It force $a_i+1\geq a'_{i+1}=a_{i+1}$ and therefore $a_i$ is not essential. We conclude that such $F_{a_i}$ is unique and hence $a_i$ is rigid.
\end{proof}

Notice that Proposition \ref{quadric rigd}-Proposition \ref{isorigid1} have proved Theorem \ref{main theorem}. As a corollary, we characterize the rigid Schubert classes in $OG(k,n)$.

\begin{proof}[Proof of Theorem \ref{rigidclass}]It is easy to see that all the conditions in Theorem \ref{rigidclass} hold if and only if all the essential sub-indices are rigid. If some essential sub-indices $a_i$ or $b_j$ are not rigid, then the proof of Proposition \ref{quadric rigid}, Proposition \ref{isonotrigid} and Proposition \ref{isorigid1} construct a non-Schubert variety representing $\sigma_{a;b}$. 

Now assume all the essential indices are rigid. Let $X$ be a subvariety representing $\sigma_{a;b}$ and let $\{F_{a_i}\}_{1\leq i\leq s}$ and $\{F_{b_j}\}_{1\leq j\leq k-s}$ be the corresponding isotropic subspaces. It suffices to show that they form a flag. 

We will use induction on $k-s$. If $k-s=1$, then the proposition reduces to Theorem \ref{rigid class in g}.

Now assume the proposition is true for $k-s<\gamma$. Let $p\in Q_X$ be a general point and define $X_p:=\{\Lambda\in X|p\in \Lambda\}$. Then $[X_p]=\sigma_{a';b'}$, where $a'_1=1$, $a'_{i+1}=a_i+1$ if $a_i\leq b_1$, $a'_{i+1}=a_i$ if $a_i>b_1$ and $b'_{j}=b_{j+1}$ for $1\leq j\leq k-s-1$. It is easy to check that if all essential indices in $(a;b)$ are rigid, then all essential indices in $(a';b')$ are also rigid. Let $\{F'_{a_i}\}_{1\leq i\leq s}$ and $\{F'_{b_j}\}_{1\leq j\leq k-s}$  be the isotropic subspaces corresponding to $X_p$. By induction on $k-s$, $\{F'_{a_i}\}_{1\leq i\leq s}$ and $\{F'_{b_j}\}_{1\leq j\leq k-s}$ form a partial flag. Notice that $F'_{a'_{i+1}}=F'_{a_i+1}$ are span of $F_{a_i}$ and $p$ for $a_i\leq b_1$, $F'_{a'_{i+1}}=F'_{a_i}$ are the span of $p$ and $F_{a_i}\cap p^\perp$ for $a_i>b_1$, $F'_{b'_j}=F'_{b_{j+1}}$ are the span of $p$ and $F_{b_{j+1}}\cap p^\perp$ for $1\leq j\leq k-s-1$. As we varying $p$, it is easy to see that $\{F_{a_i}\}_{1\leq i\leq s}$ and $\{F_{b_j}\}_{1\leq j\leq k-s}$ have also to form a partial flag. We conclude that the Schubert class is rigid if all essential indices are rigid.
\end{proof}

\end{document}